\documentclass[12pt, a4paper]{amsart}
\usepackage{amssymb}
\usepackage{tikz} 
\usepackage{mathabx}
\usepackage{graphicx}
\usepackage{mathrsfs}
\usepackage{hyperref}
\usepackage{import}
\usepackage{tcolorbox}

\theoremstyle{plain}

\newtheorem{theorem}{Theorem}[section]
\newtheorem{lemma}[theorem]{Lemma}
\newtheorem{proposition}[theorem]{Proposition}
\newtheorem{corollary}[theorem]{Corollary}

\theoremstyle{definition}
\newtheorem{define}{Definition}[section]

\newtheorem{example}{Example}[section]

\theoremstyle{remark}
\newtheorem{remark}{Remark}[section]

\newtheorem*{acknowledgement}{Acknowledgement}

\begin{document}

\date\today

\title[Pinchuk scaling method]{Pinchuk scaling method on domains with non-compact automorphism groups}
\author{Ninh Van Thu, Nguyen Thi Kim Son and Nguyen Quang Dieu\textit{$^{1,2}$}}

\address{Ninh Van Thu}
\address{ School of Applied Mathematics and Informatics, Hanoi University of Science and Technology, No. 1 Dai Co Viet, Hai Ba Trung, Hanoi, Vietnam}
\email{thu.ninhvan@hust.edu.vn}

\address{Nguyen Thi Kim Son}
\address{~Department of Mathematics, Hanoi University, Nguyen Trai, Nam Tu Liem, Hanoi, Vietnam}
\email{kimsonnt.0611@gmail.com}

\address{Nguyen Quang Dieu}
\address{\textit{$^{1}$}~Department of Mathematics, Hanoi National University of Education, 136 Xuan Thuy, Cau Giay, Hanoi, Vietnam}
 \address{\textit{$^{2}$}~Thang Long Institute of Mathematics and Applied Sciences,
Nghiem Xuan Yem, Hoang Mai, HaNoi, Vietnam}
\email{ngquang.dieu@hnue.edu.vn}

\subjclass[2020]{Primary 32M05; Secondary 32H02, 32F18.}
\keywords{Automorphism group, scaling method, $h$-extendible domain}

\begin{abstract} In this paper, we characterize weakly pseudoconvex domains of finite type in $\mathbb C^n$ in terms of the boundary behavior of automorphism orbits by using the scaling method.
\end{abstract}
\maketitle

\section{introduction}

Let $\Omega$ be a domain in $\mathbb C^n$. The set of all automorphisms of $\Omega$, denoted by $\mathrm{Aut}(\Omega)$, makes a group under composition and this group is also a topological group with the topology of uniform convergence on compact sets of $\Omega$. By a classical theorem of H. Cartan (see \cite{Na71}), for a bounded domain $\Omega$ in $\mathbb C^n$, it follows that  $\mathrm{Aut}(\Omega)$ is non-compact if and only if  there exist a point $a \in \Omega$, a point $p \in \partial \Omega$, and automorphisms $\varphi_j \in \mathrm{Aut}\Omega)$ such that $\varphi_j(a) \to p$ as $j\to \infty$. Such a point $p$ is called a {\it boundary orbit accumulation point}. The local geometry of the boundary orbit accumulation point in turn gives global information about the characterization of domains. In particular, Greene and Krantz \cite{GK93} posed a conjecture that for a smoothly bounded pseudoconvex domain admitting a non-compact automorphism group, the boundary orbit accumulation point is of finite type in the sense of D'Angelo \cite{D'A82}. (In this paper, the finiteness of type is understood in the sense of D'Angelo.) The interested reader is referred to the recent papers \cite{IK99, KN15, Kr21} for this conjecture.

In this paper, we study the problem of characterizing domains in $\mathbb C^n$ with non-compact automorphism groups. The main results around this problem are due to B. Wong and J. P. Rosay \cite{Wo77, Ro79}, E. Bedford and S. Pinchuk \cite{BP89, BP91, BP95, BP99}, K.-T. Kim \cite{Ki90}, F. Berteloot \cite{Be94,Be03}, A. Isaev and S. Krantz \cite{IK97}, Do Duc Thai and the first author \cite{DN09}, the first and third authors \cite{NN20}. Almost all previous work requires the finiteness of type and either the strong pseudoconvexity (or even convexity), or pseudoconvexity only in dimension $2$. In contrast to these results, we provide a new characterization of weakly pseudoconvex domains of finite type in terms of the boundary behavior of automorphism orbits by using the scaling method, introduced by S. Pinchuk \cite{Pi81}. 

The scaling method may be briefly described as follows. Let $\Omega$ be a domain in $\mathbb C^{n+1}$ and $\{\varphi_j(a)\}$ be a sequence of automorphism orbits  converging to a boundary point $\xi_0$. Let us fix a small neighborhood $U_0$ of $\xi_0$. By using the reasonable composition, say $T_j$, of  polynomial automorphisms of $\mathbb C^{n+1}$, including translations and dilations, the sequence of domains $D_j:=T_j(U_0\cap \Omega)$ converges normally to a model $M_P$, given by
$$
M_P:=\left\{(z,w)\in \mathbb C^n\times \mathbb C\colon \mathrm{Re}(w)+P(z,\bar z)<0\right\},
$$
where $P$ is a real-valued polynomial on $\mathbb C^n$. When $P$ is a homogeneous plurisubharmonic polynomial, $M_P$ is called a \emph{local homogeneous model} at $\xi_0$. 

The sequence $F_j:=T_j\circ \varphi_j$ is in turn called a \emph{Pinchuk scaling sequence}. The most difficulty is to prove that the Pinchuk scaling sequence $\{F_j\}$ is normal, i.e., there exists a subsequence of $\{F_j\}$ that converges uniformly on compacta from $\Omega$ to $M_P$. We first assume, temporarily, that $\Omega$ is a bounded domain in $\mathbb C^n$. Then the approach of Bedford and Pinchuk conveniently splits into two steps. In the first step, E. Bedford and S. Pinchuk \cite{BP89, BP91, BP95, BP99} considered, alternatively, the convergence of the backward scaling sequence $\{F_j^{-1}\}$. Thanks to the boundedness of $\Omega$, the Montel theorem ensures that this sequence $\{F_j^{-1}\}$ contains a convergent subsequence.  Next, that limit, say $G$, was showed to be one-one from $M_P$ into $\Omega$ by using the uniform estimates of the Kobayashi metric (cf. \cite{Cat89} for $n=1$), or the Sibony metric (cf. \cite{Sib81} for corank one domains) on a family of the scaling domains $\{D_j\}$. In addition, the existence of a plurisubharmonic exhaustion function for $\Omega$ (cf. \cite{DF77}) yields the holomorphic map $G$ is surjective. In the second step, they treated a one-dimensional subgroup $\{h_t\}_{t\in\mathbb R}\subset \mathrm{Aut}(\Omega)$ defined by $h_t(z)=G\left(G(z)+(0',it)\right), \; z\in \Omega$.  This subgroup is parabolic in the sense that $h_t(z)$  tends to some boundary point $p\in \partial \Omega$ as $j\to \pm \infty$ for any $z\in\Omega$. A careful analysis of the holomorphic vector field $ H(z):=\frac{d}{dt}\mid_{t=0}h_t(z)$, defined on $\Omega$ and tangent to $\partial \Omega$ (by \cite{Fe74} each $h_t, t\in \mathbb R$, extends smoothly to the boundary), at the parabolic fixed point $\xi_0$ shows that the polynomial $P$ is a weighted homogeneous polynomial such as $P(z_1,\bar z_1)=c|z_1|^{2m}$ for $n=1$ and $P(z,\bar z)=c|z_1|^{2m}+|z_2|^2+\cdots +|z_n|^2$ for corank one domains, where $c$ is a positive constant.
  
Let us emphasize that the above-mentioned method does not work for unbounded domains. Therefore, an alternative approach is to prove directly the normality of $\{F_j\}$ and then the tautness of $\Omega$ indicates that $\{F_j^{-1}\}$ is also normal. Then, \cite[Lemma 4.1]{GK87}(see also \cite[Prop. 2.1]{DN09}) guarantees that the limit of $\{F_j\}$ is a biholomorphism from $\Omega$ onto $M_P$. The tautness of $\Omega$ easily follows from the existence of a plurisubharmonic peak function at $\xi_0$ (cf. \cite[Prop. 2.1]{Be94}). Therefore, our work boils down to verify the normality of $\{F_j\}$.
 
In 1991, S. Pinchuk \cite{Pi91} himself considered the case that $\xi_0$ is strongly pseudoconvex. Thanks to the locally convexifiability of $\Omega$ near $\xi_0$, S. Pinchuk proved that $\{F_j\}$ is normal. Therefore, $\Omega$ is biholomorphically equivalent to the unit ball $\mathbb B^n$, which gives a local version of  the Wong-Rosay theorem. Similar result was achieved by A. M. Efimov \cite{Ef95} for unbounded strongly pseudoconvex domains in $\mathbb C^n$. In addition, for convex domains in $\mathbb C^n$ the normality of our Pinchuk scaling sequence can instead be easily established (cf. \cite{BP95, Ga97, Ni09}). However, the Frankel scaling method given in \cite{Fra89} can be applied for convenience (cf. \cite{Ki90, Ki01, Zi17, Jo18}).

For unbounded weakly pseuconvex domains of finite type in $\mathbb C^2$, F. Berteloot \cite{Be94, Be03, Be06} obtained a significant progress by using the properties of polydics constructed by D. Catlin (see \cite{Cat89}) and the corresponding estimate of Kobayashi metric near $\xi_0$ to show the normality of the Pinchuk scaling sequence. Hence, $\Omega$ is biholomorphically equivalent to some local homogeneous model $M_P$. This result was generalized by Do Duc Thai and the first author \cite{DN09} for corank one domains in $\mathbb C^{n}$.

Recently, the first and third authors \cite{NN20} investiaged a pseudoconvex domain $\Omega\subset\mathbb C^n$ which is of finite Catlin's multitype near $\xi_0\in \partial \Omega$. When $M_P$ is a homogeneous model of finite type, there exists a plurisubharmonic peak function for $M_P$ at the origin. Therefore, the attraction property of analytic discs yields the normality of the Pinchuk scaling sequence. Consequently, $\Omega$ is biholomorphically equivalent to $M_P$ provided that the automorphism orbit $\{\varphi_j(a)\}$ converges $\Lambda$-nontangentially to $\xi_0$ (cf. Remark \ref{remark-lambda-nontangent} and \cite[Theorem $1.1$]{NN20}). 

One observes that the local homogeneous model $M_P$ depends heavily on the boundary behavior of $\{\varphi_j(a)\}$. More precisely, the boundary behavior of $\{\varphi_j(a)\}$ suggests some choice of dilations. The purpose of this paper is to give a characterization of local homogeneous models when the automorphism orbit $\{\varphi_j(a)\}$ accumulates at $\xi_0$ ``very" tangentially to $\partial \Omega$ - the remaining possibility (cf. Definitions \ref{lambda-tangent} and \ref{spherically-convergence}).

The first aim of this paper is to prove the following theorem, which says that only the unit ball admits an automorphism orbit accumulating at $\xi_0$ uniformly $\Lambda$-tangentially to $\partial\Omega$ (cf. Definition  \ref{lambda-tangent}).
\begin{theorem}\label{maintheorem1}
Let $\Omega$ be a pseudoconvex domain in $\mathbb C^{n+1}$ with $\mathcal{C}^\infty$-smooth boundary $\partial \Omega$. Let $\xi_0\in\partial \Omega$ be strongly $h$-extendible with Catlin's finite multitype $(2m_1,\ldots,2m_n, 1)$ and let $\Lambda=(1/2m_1, \ldots, 1/2m_n)$ (see Definition \ref{strongly-h-extendible}). Suppose that there exists a sequence $\{\varphi_j\}\subset \mathrm{Aut}(\Omega)$ such that $\{\varphi_j(a)\}$ converges uniformly $\Lambda$-tangentially to $\xi_0$ for some $a\in \Omega$ (see Definition  \ref{lambda-tangent}).
Then $\Omega$ is biholomorphically equivalent to the unit ball $\mathbb B^{n+1}$.
\end{theorem}
\begin{remark} The uniform $\Lambda$-tangential convergence of $\{\varphi_j(a)\}$ allows us to choose a suitable sequence of dilations (cf. see Equation (\ref{dilationj}) in Section \ref{S-h-extendible}) so that our model is an analytic ellipsoid that is biholomorphically equivalent to $\mathbb B^{n+1}$. However, Example \ref{ex3.1} in Section \ref{S-h-extendible} points out that without this uniform $\Lambda$-tangential convergence, there still exists an alternative sequence of dilations to get such a model. In addition, the explicit description for the automorphism group of the Thullen domains or the finite multitype models (cf. \cite{NNTK19}) demonstrates that any sequence of automorphism orbit converges $\Lambda$-nontangentially to some boundary orbit accumulation point. Therefore, it seems reasonable to expect that may drop the requirement of the uniformity of  $\Lambda$-tangentially convergences (see Remark \ref{model-behavior}).
\end{remark}

Now we turn to pseudoconvex Levi corank one domains in $\mathbb C^{n+1}$ which includes pseudoconvex domains of finite type in $\mathbb C^2$.  Then, the point $\xi_0$ is $h$-extendible (cf. \cite{Yu95}). In addition, if  $n=1$ and $\xi_0$ is strongly $h$-extendible, then $\Omega$ is biholomorphically equivalent to $\mathbb B^{2}$ by Theorem \ref{maintheorem1}. However, without the strongly $h$-extendibility, the notion of spherically $\frac{1}{2m}$-tangential convergence is necessary to determine if $\Omega$ is biholomorphically equivalent to $\mathbb B^{n+1}$ (cf. Definition \ref{spherically-convergence}).
 
 The second aim of this paper is to prove the following theorem.
 \begin{theorem}\label{maintheorem3}
Let $\Omega$ be a pseudoconvex domain in $\mathbb C^{n+1}$ with $\mathcal{C}^\infty$-smooth boundary $\partial \Omega$. Suppose that $\xi_0$ is a boundary point of $\Omega$ of D'Angelo finite type such that the Levi form has corank at most $1$ at $\xi_{0}$ and  there exists a sequence $\{\varphi_j\}\subset \mathrm{Aut}(\Omega)$ such that $\varphi_j(a)$ converges spherically $\frac{1}{2m}$-tangentially to $\xi_0$ for some $a\in \Omega$ (cf. Definition \ref{spherically-convergence}). Then $\Omega$ is biholomorphically equivalent to the unit ball $\mathbb B^{n+1}$.
\end{theorem}

Now let $\Omega\subset \mathbb C^2$ be a pseudoconvex domain of finite type  near $\xi_0\in \partial\Omega$ with the type $2m$. Then the notion of uniformly $(\frac{1}{2m})$-tangential convergence (cf. Definition  \ref{lambda-tangent}) reduces to that of $(\frac{1}{2m})$-tangential convergence. Moreover, the notion of spherically $\frac{1}{2m}$-tangential convergence is exactly Definition \ref{spherically-convergence}. Therefore, Theorem \ref{maintheorem3} yields the following corollary.
\begin{corollary}\label{maintheorem2}
Let $\Omega$ be a pseudoconvex domain in $\mathbb C^{2}$ with $\mathcal{C}^\infty$-smooth boundary $\partial \Omega$. Suppose that $\xi_0\in\partial \Omega$ is of finite type $2m$. Suppose that there exists a sequence $\{\varphi_j\}\subset \mathrm{Aut}(\Omega)$ such that $\varphi_j(a)$ converges spherically $\frac{1}{2m}$-tangentially to $\xi_0$ for some $a\in \Omega$ (see Definition \ref{spherically-convergence}).
Then $\Omega$ is biholomorphically equivalent to the unit ball $\mathbb B^{2}$.
\end{corollary}

In the case that $\{\varphi_j(a)\}$ does not converge spherically $\frac{1}{2m}$-tangentially to $\xi_0$ for some $a\in \Omega$, we prove the following proposition, in which our model may be defined by a homogeneous polynomial of degree larger than $2$.
\begin{proposition} \label{1.2new}
Let $\Omega$ be a pseudoconvex domain in $\mathbb C^{2}$ with $\mathcal{C}^\infty$-smooth boundary $\partial \Omega$. Suppose that $\xi_0\in\partial \Omega$ is of finite type $2m$. Suppose that there exist a number $2 \leq \nu \leq m$, $a \in \Omega$ and a sequence
$f_j \subset Aut (\Omega)$ such that $f_j (a)$ converges spherically $\frac{1}{2m}$-tangentially of order $2\nu$ to $\xi_0$ (see Definition \ref{spherically-convergence-2}). Then $\Omega$ is biholomorphically equivalent to a model of the form 
$$
M_{Q}:=\left \{(z, w)\in \mathbb C^2\colon\mathrm{Re}(w)+Q(z)<0\right\},
$$
where $Q$ is a homogeneous polynomial of degree $2\nu$ which is not harmonic.
\end{proposition}
\begin{remark} We note that Example \ref{KNex} in Section \ref{Sec5} illustrates that if the sequence of automorphism orbits does not converge spherically $\frac{1}{2m}$-tangentially to a boundary point, then our domain $\Omega$ is not biholomorphically equivalent to the unit ball $\mathbb B^2$ but to $M_Q$ with $\deg(Q)=4$.
\end{remark}

The organization of this paper is as follows: In Sections \ref{technical-section}, we recall some basic definitions and results needed later. In Section \ref{S-h-extendible}, we present the notion of $\Lambda$-tangential convergence and give a proof of Theorem \ref{maintheorem1}. Next, the notion of spherical $\frac{1}{2m}$-tangential convergence and the proof of Theorem \ref{maintheorem3} are introduced in Section \ref{S-corank1}. Finally, the proof of Proposition \ref{1.2new} is given in Section \ref{Sec5}.

\section{Preliminaries}\label{technical-section}
First of all, we recall the following definition (see \cite{GK87, Kr21}, or \cite{DN09}). 
\begin{define} Let $\{\Omega_i\}_{i=1}^\infty$ be a sequence of domains in $\mathbb C^n$. The sequence $\{\Omega_i\}_{i=1}^\infty$ is said to \emph{converge normally} to a domain $\Omega_0\subset \mathbb C^n$ if the following two conditions hold:
	\begin{enumerate}
		\item[(i)] If a compact set $K$ is contained in the interior (i.e., the largest open subset) of  $\displaystyle\bigcap_{j\geq m} \Omega_j$ for some positive integer $m$, then $K\subset \Omega_0$.
		\item[(ii)] If a compact subset $K'\subset \Omega_0$ , then there exists a constant $m>0$ such that $\displaystyle K'\subset \bigcap_{j\geq m} \Omega_j$.
	\end{enumerate}  
In addition, when a sequence of  map $\varphi_j\colon \Omega_j\to\mathbb C^m$ converges uniformly on compact sets to a map $\varphi_j\colon \Omega\to\mathbb C^m$ then we shall say that $\varphi_j$ \emph{converges normally} to $\varphi$. 
\end{define}

Next, let us recall the following definition (cf. \cite{Yu95}).
\begin{define} Let $\Lambda=(\lambda_1,\ldots,\lambda_n)$ be a fixed $n$-tuple of positive numbers and $\mu>0$. We denote by $\mathcal{O}(\mu,\Lambda)$ the set of smooth functions $f$ defined near the origin of $\mathbb C^n$ such that
$$
D^\alpha \overline{D}^\beta f(0)=0~\text{whenever}~ \sum_{j=1}^n (\alpha_j+\beta_j)\lambda_j \leq \mu.
$$
If $n=1$ and $\Lambda = (1)$ then we use $\mathcal{O}(\mu)$ to denote the functions vanishing to order at least $\mu$ at the origin. Here and in what follows, $D^\alpha$ and $\overline{D}^\beta$ denote the partial differential operators
\[
\frac{\partial^{|\alpha|}}{\partial z_1^{\alpha_1}\cdots \partial z_n^{\alpha_n }}~\text{and}~\frac{\partial^{|\beta|}}{\partial \bar z_1^{\beta_1}\cdots \partial \bar z_n^{\beta_n }},
\]
respectively. Furthermore, $\lesssim$ and $\gtrsim$ denote inequality up to a positive constant. Moreover, we will use $\approx $ for the combination of $\lesssim$ and $\gtrsim$

\end{define}

Finally, in order to give proofs of Theorem \ref{maintheorem1} and Theorem \ref{maintheorem2}, let us recall the following proposition that is the main ingredient  in our argument (see \cite[Prop. 4.3]{NN20}).
\begin{proposition}[\cite{NN20}]\label{pro-scaling} Let $\omega$ be a domain in $\mathbb C^k$, $a\in \omega$ and $\sigma_j\colon \omega \to \Omega_j$ be a sequence of holomorphic mappings such that $\{\sigma_j(a)\}\Subset M_P$. If $M_P$ is of finite type, then $\{\sigma_j\}$ contains a subsequence that converges locally uniformly to a holomorphic map $\sigma\colon \omega \to M_P$. 
\end{proposition}
\section{The behaviour of automorphism orbits accumulating at a boundary point of an $h$-extendible domain in $\mathbb C^n$}\label{S-h-extendible}
\subsection{$\Lambda$-tangential convergence}\label{Ss3.1}
Throughout this subsection, let $\Omega$ be a domain in $\mathbb C^n$ and assume that $\xi_0\in \partial \Omega $ is an $h$-extendible boundary point (cf. \cite{Yu95, DH94}). Let $\rho$ be a local defining function for $\Omega$ near $\xi_0$ and let the multitype $\mathcal{M}(\xi_0)=(2m_1,\ldots,2m_n,1)$ be finite (see \cite{Cat84}). (Note that because of the pseudoconvexity of $\Omega$, the integers $2m_1,\ldots,2m_n$ are all even.) Let us denote by $\Lambda=(1/2m_1,\ldots,1/2m_n)$. By the definition of multitype, there are distinguished coordinates $(z,w)=( z_1, \ldots,z_n,w)$ such that $\xi_0=0$ and $\rho(z,w)$ can be expanded near $0$ as follows:
$$
\rho(z,w)=\mathrm{Re}(w)+P(z)+Q(z,w),
$$ 
 where $P$ is a $\Lambda$-homogeneous plurisubharmonic polynomial that contains no pluriharmonic monomials, $Q$ is smooth and satisfies 
 $$
 |Q(z,w)|\leq C \left( |w|+ \sum_{j=1}^n |z_j|^{2m_j} \right)^\gamma,
 $$ 
 for some constant $\gamma>1$ and $C>0$. 
In what follows, $\mathrm{dist}(z,\partial\Omega)$ denotes the Euclidean distance from $z$ to $\partial\Omega$.
\begin{define}\label{lambda-tangent}
We say that a sequence $\{\eta_j=(\alpha_j,\beta_j)\}\subset  \Omega$ with $\alpha_j=(\alpha_{j 1},\ldots,\alpha_{j n})$, \emph{converges uniformly $\Lambda$-tangentially to $\xi_0$} if the following conditions hold:
\begin{itemize}
\item[(a)] $|\mathrm{Im}(\beta_j)|\lesssim |\mathrm{dist}(\eta_j,\partial \Omega)|$;
\item[(b)] $|\mathrm{dist}(\eta_j,\partial \Omega)|=o(|\alpha_{jk}|^{2m_k})$ for $1\leq k\leq n$;
\item[(c)] $|\alpha_{j1}|^{2m_1}\approx |\alpha_{j2}|^{2m_2}\approx \cdots\approx |\alpha_{jn}|^{2m_n}$.
\end{itemize}
\end{define}
\begin{remark} \label{remark-lambda-nontangent}It is well-known that $\{\eta_j\}\subset \Omega$ converges nontangentially to $\xi_0$ if $|\mathrm{Im}(\beta_j)|\lesssim |\mathrm{dist}(\eta_j,\partial \Omega)|$ and $|\alpha_{j k}|\lesssim|\mathrm{dist}(\eta_j,\partial \Omega)|$ for every $1\leq k\leq n$. Nevertheless, such sequence converges $\Lambda$-nontangentially to $\xi_0$ if $|\mathrm{Im}(\beta_j)|\lesssim |\mathrm{dist}(\eta_j,\partial \Omega)|$ and $|\alpha_{j k}|^{2m_k}\lesssim|\mathrm{dist}(\eta_j,\partial \Omega)|$ for every $1\leq k\leq n$ (cf. \cite{NN20}).
\end{remark}
Denote by
$$
\sigma(z):=\sum_{k=1}^n |z_k|^{2m_k}.
$$ 
\begin{define}\label{strongly-h-extendible} We say that a boundary point $\xi_0\in \partial \Omega$ is \emph{strongly $h$-extendible} if there exists $\delta>0$ such that $P(z)-\delta \sigma(z)$ is plurisubharmonic, i.e. $dd^c P\geq \delta dd^c \sigma$. 
\end{define}
\begin{remark}\label{remark-strongly-h-extendible} The notion of strongly $h$-extendibility is exactly that $M_P$ is \emph{homogeneous finite diagonal type} given in \cite{He92, He16}. Let $\xi_0\in \partial \Omega$ be strongly $h$-extendible. Then by $dd^c P\gtrsim  dd^c \sigma$, it follows that $\xi_0\in \partial \Omega$ is in fact  $h$-extendible and
\begin{align*}
\sum_{k, l=1}^n \frac{\partial^2P}{\partial z_k\partial \bar z_l} (\alpha) w_j\bar w_l&\gtrsim \sum_{k, l=1}^n \frac{\partial^2\sigma}{\partial z_k\partial \bar z_l} (\alpha) w_j\bar w_l\\
&\gtrsim m_1^2|\alpha_1|^{2m_1-2}|w_1|^2+\cdots+m_n^2|\alpha_n|^{2m_n-2}|w_n|^2
\end{align*}
for all $\alpha,w\in \mathbb C^n$. 
\end{remark}

In the sequel, we will assume that $\xi_0\in \partial\Omega$ is a strongly $h$-extendible point and let $\{\epsilon_j\}\subset \mathbb R^+$ be a given sequence. Then we define the sequence $\tau_j=(\tau_{j1},\ldots,\tau_{jn})$, associated to $\{\epsilon_j\}$, as follows:
\begin{equation*}
\tau_{jk}:=|\alpha_{k}|.\left(\dfrac{\epsilon_j}{|\alpha_{jk}|^{2m_k}}\right)^{1/2},\; j\geq 1, 1\leq k\leq n.
\end{equation*}
A simple calculation shows that $\tau_{jk}^{2m_k}=\epsilon_j.\left(\frac{\epsilon_j}{|\alpha_{jk}|^{2m_k}}\right)^{m_k-1}\lesssim \epsilon_j$. Hence, we get the following estimates
\begin{align}\label{tau-estimate}
\epsilon_j^{1/2}\lesssim \tau_{jk}\lesssim \epsilon_j^{1/2m_k}.
\end{align}

In what follows, we assign weights $\frac{1}{2m_1}, \ldots,\frac{1}{2m_{n}}, 1$ to the variables $z_1,\ldots, z_{n}, w$, respectively and denote by $wt(K):=\sum_{j=1}^{n} \frac{k_j}{2m_j}$ the weight of an $n$-tuple $K=(k_1, \ldots, k_{n})\in \mathbb Z^{n}_{\geq0}$. We note that  $wt(K+L)=wt(K)+wt(L)$ for any $K, L\in \mathbb Z^{n}_{\geq0}$.  

In order to prove Theorem \ref{maintheorem1}, we need the following lemmas. First of all, from (\ref{tau-estimate}) one easily obtains the following lemma.
 \begin{lemma}\label{higher-weight} Let $f(z,w)$ be a $\mathcal{C}^\infty$-smooth function defined in a neighborhood of the origin in $ \mathbb C^{n+1}$ vanishing to weight order greater than $1$ at the origin. Then 
 $$
 f(\tau_{j1} z_1,\ldots,\tau_{jn} z_n,\epsilon_j w)=o(\epsilon_j).
 $$
 \end{lemma}
 For monomials with weight order $\leq 1$, we have the following lemmas.
 \begin{lemma}\label{Cn-spherical-convergence} Let $p,q\in \mathbb N^n$ be two multi-indices. Then, for all polynomials $P$ one has
$$
\epsilon_j^{-1}\left |D^{p} \overline{D}^q P(\alpha_j)\tau_j^{p+q}\right|\to 0
$$
for $|p|+|q|>2$. In addition, if $|p|=|q|=1$, then 
$$
\epsilon_j^{-1}\left |D^{p} \overline{D}^q P(\alpha_j)\tau_j^{p+q}\right|\lesssim 1.
$$
Moreover, if $P(z)-\delta \sigma(z)$ is plurisubharmonic for some $\delta>0$, then
\begin{align*}
\epsilon_j^{-1}\sum_{k, l=1}^n \frac{\partial^2P}{\partial z_k\partial \bar z_l} (\alpha_j)\tau_{jk}\tau_{jl} w_k\bar w_l\gtrsim m_1^2|w_1|^2+\cdots+m_n^2|w_n|^2.
\end{align*}
\end{lemma}
\begin{proof} It suffices to prove the lemma for $P(z)=z^K\bar z^L$ and $K\geq p, L\geq q$. Then we have
\begin{align*}
\epsilon_j^{-1}\left |D^{p} \overline{D}^q P(\alpha_j)\tau_j^{p+q}\right|&= \epsilon_j^{-1} |\alpha_{j1}|^{k_1+l_1}\left(\dfrac{\tau_{j1}}{|\alpha_{j1}|}\right)^{p_1+q_1}\cdots |\alpha_{jn}|^{k_n+l_n}\left(\dfrac{\tau_{jn}}{|\alpha_{jn}|}\right)^{p_n+q_n}\\
&= \left(\dfrac{|\alpha_{j1}|^{2m_1}}{\epsilon_j}\right)^{\frac{k_1+l_1}{2m_1}-\frac{p_1+q_1}{2}}\cdots \left(\dfrac{|\alpha_{jn}|^{2m_n}}{\epsilon_j}\right)^{\frac{k_n+l_n}{2m_n}-\frac{p_n+q_n}{2}}\\
&\approx \left(\dfrac{|\alpha_{j1}|^{2m_1}}{\epsilon_j}\right)^{\sum\limits_{s=1}^n\frac{k_s+l_s}{2m_s}-\frac{p_j+q_j}{2}}=\left(\dfrac{|\alpha_{j1}|^{2m_1}}{\epsilon_j}\right)^{1-\frac{|p|+|q|}{2}}.
\end{align*}
Therefore,  we get
$$
\epsilon_j^{-1}\left |D^{p} \overline{D}^q P(\alpha_j)\tau_j^{p+q}\right|\to 0
$$
as $j\to\infty$ for $|p|+|q|>2$ and 
$$
\epsilon_j^{-1}\left |D^{p} \overline{D}^q P(\alpha_j)\tau_j^{p+q}\right|\lesssim 1
$$
for $|p|+|q|=2$. Finally, by Remark \ref{remark-strongly-h-extendible} one obtains that
\begin{align*}
\epsilon_j^{-1}\sum_{k, l=1}^n \frac{\partial^2P}{\partial z_k\partial \bar z_l} (\alpha_j)\tau_{jk}\tau_{jl} w_k\bar w_l&\gtrsim \epsilon_j^{-1}\sum_{k, l=1}^n \frac{\partial^2\sigma}{\partial z_k\partial \bar z_l} (\alpha_j) \tau_{jk}\tau_{jl} w_j\bar w_l\\
&\gtrsim\epsilon_j^{-1}\left(m_1^2|\alpha_1|^{2m_1-2}\tau_{j1}^2|w_1|^2+\cdots+m_n^2|\alpha_n|^{2m_n-2}\tau_{jn}^2 |w_n|^2\right)\\
&\gtrsim m_1^2|w_1|^2+\cdots+m_n^2|w_n|^2
\end{align*}
for every $w\in \mathbb C^n$.
\end{proof}
In the same fashion we have the following lemma.
\begin{lemma}\label{Cn-spherical-convergence-higher-order} Let $Q(z)$ be a polynomial in $z\in \mathbb C^n$ such that $Q\in \mathcal{O}(1,\Lambda)$. Then we have
$$
\epsilon_j^{-1}\left |D^{p} \overline{D}^q Q(\alpha_j)\tau_j^{p+q}\right|\to 0
$$
as $j\to\infty$ for $|p|+|q|\geq 2$. 
\end{lemma}
\begin{proof} As in the proof of Lemma \ref{Cn-spherical-convergence}, it suffices to consider $Q(z)=z^K\bar z^L$ and $K\geq p, L\geq q$ with $d:=wt(K+L)>1$. Then following the proof of Lemma \ref{Cn-spherical-convergence}, one has
$$
\epsilon_j^{-1}\left |D^{p} \overline{D}^q Q(\alpha_j)\tau_j^{p+q}\right|\approx\left(\dfrac{|\alpha_{j1}|^{2m_1}}{\epsilon_j}\right)^{d-\frac{|p|+|q|}{2}}.
$$
Therefore, we conclude that $\epsilon_j^{-1}\left |D^{p} \overline{D}^q Q(\alpha_j)\tau_j^{p+q}\right|\to 0$ as $j\to\infty$ for $|p|+|q|\geq 2$, as desired.
\end{proof}

\subsection{Proof of Theorem \ref{maintheorem1}}
Let $\Omega$ and $\xi_0\in \partial \Omega $ be as in the statement of Theorem \ref{maintheorem1}. Let $\mathcal{M}(\xi_0)=(2m_1,\ldots,2m_n,1)$ be the finite multitype of $\Omega$ at $\xi_0$ and denote by $\Lambda=(1/2m_1,\ldots,1/2m_n)$. As in Subsection \ref{Ss3.1}, one can find local coordinates $(\tilde z,\tilde w)=(\tilde z_1,\ldots,\tilde z_n,\tilde w)$ near $\xi_0$ such that $\xi_0=0$ and the local defining function $\rho(\tilde z,\tilde w)$ for $\Omega$ can be expanded near $0$ as follows:
$$
\rho(\tilde z,\tilde w)=\mathrm{Re}(\tilde w)+P(\tilde z)+Q(\tilde z,\tilde w),
$$ 
 where $P$ is a $\Lambda$-homogeneous plurisubharmonic polynomial that contains no pluriharmonic monomials, $Q$ is smooth and satisfies 
 $$
 |Q(\tilde z,\tilde w)|\leq C \left( |\tilde w|+ \sum_{j=1}^n |\tilde z_j|^{2m_j} \right)^\gamma,
 $$ 
 for some constant $\gamma>1$ and $C>0$. 
 
 By hypothesis of Theorem \ref{maintheorem1}, there exist a sequence $\{\varphi_j\}\subset \mathrm{Aut}(\Omega)$ and a point $a\in \Omega$ such that 
$\eta_j:=\varphi_j(a)$ converges uniformly $\Lambda$-tangentially to $\xi_0$. Let us write $\eta_j=(\alpha_j,\beta_j)=(\alpha_{j1},\ldots,\alpha_{jn},\beta_j)$. Then one has 
\begin{itemize}
\item[(a)] $|\mathrm{Im}(\beta_j)|\lesssim |\mathrm{dist}(\eta_j,\partial \Omega)|$;
\item[(b)] $|\mathrm{dist}(\eta_j,\partial \Omega)|=o(|\alpha_{jk}|^{2m_k})$ for $1\leq k\leq n$.
\item[(c)] $|\alpha_{j1}|^{2m_1}\approx |\alpha_{j2}|^{2m_2}\approx \cdots\approx |\alpha_{jn}|^{2m_n}$.
\end{itemize}
By following the proofs of Lemmas $4.10$, $4.11$ in \cite{Yu95}, after a change of variables
\[\begin{cases}
z:=\tilde z;\\
 w:=\tilde w+ b_1(\tilde z)\tilde w+b_2(\tilde z)\tilde w^2+b_3(\tilde z),
\end{cases}
\]
where $b_1,b_2, b_3$ are holomorphic functions of $\tilde z$ satisfying $b_k=O(|\tilde z|^2)$, $k=1,2,3$, there are local holomorphic coordinates $(z,w)$ in which $\xi_0=0$ and $\Omega$ can be described near $0$ as follows: 
 $$
 \Omega=\left\{(z,w)\in \mathbb C^{n+1}\colon\rho(z,w)=\mathrm{Re}(w)+ P(z) +R_1(z) + R_2(\mathrm{Im} w)+(\mathrm{Im} w) R(z)<0\right\},
 $$ 
 where $R_1\in \mathcal{O}(1, \Lambda),R\in \mathcal{O}(1/2, \Lambda) $, and $R_2\in \mathcal{O}(2)$. We would like to emphasize that in the new coordinates the sequence $\{\eta_j\}$ still has the properties $\mathrm{(a)}, \mathrm{(b)}$, and $\mathrm{(c)}$. 

Let $U_0$ be a fixed small neighborhood of $\xi_0=0$. Then for any sequence $\{\eta_j=(\alpha_j,\beta_j)\}$ of points converging uniformly $\Lambda$-tangentially to the origin in $U_0\cap\{\rho<0\}=:U_0^-$, we associate with a sequence of points $\eta_j'=(\alpha_{j}, a_j +\epsilon_j+i b_j)$, where $\epsilon_j>0$ and $\beta_j=a_j+i b_j$, such that $\eta_j'=(\alpha_j,\beta_j')$ with $\beta_j'= a_j +\epsilon_j+i b_j$ is in the hypersurface $\{\rho=0\}$ for every $j\in\mathbb N^*$. We note that $\epsilon_j\approx \mathrm{dist}(\eta_j,\partial \Omega)$ 

Before we begin the scaling procedure, we make several changes of coordinates. Firstly, let us consider the sequences of  translations $L_{\eta_j'}\colon \mathbb C^n\to\mathbb C^n$ , defined by
$$
L_{\eta_j'}(z,w):=(z,w)-\eta_j'=(z-\alpha_j,w-\beta_j').
$$
Then, under the change of variables $(\tilde z,\tilde w):=L_{\eta_j'}(z,w)$, i.e.,
\[
\begin{cases}
w-\beta_j'= \tilde{w};\\
z_k-\alpha_{j k}=\tilde{z}_k,\, k=1,\ldots,n,
\end{cases}
\]
one observes that $L_{\eta_j'}(\alpha_j,\beta_j)=(0',-\epsilon_j)$ for every $j\in \mathbb N^*$. 

Now let us write $w=u+iv, v=b_j+(v-b_j)=b_j+\mathrm{Im}(\tilde w)$, and $z=\alpha_j+(z-\alpha_j)=\alpha_j+\tilde z$. Since $R_2\in \mathcal{O}(2)$ and $R\in \mathcal{O}(1/2, \Lambda) $, by using Taylor's theorem, we have
\begin{equation*}
\begin{split}
R_2(v)&=R_2(b_j)+R_2'(b_j) (v-b_j) +o(|v-b_j|)=R_2(b_j)+R_2'(b_j) \mathrm{Im}(\tilde w) +o(|\mathrm{Im}(\tilde w)|);\\
vR(z)&=(b_j+(v-b_j))R(\alpha_j+\tilde z)\\
&=(b_j+\mathrm{Im}(\tilde w)) \Big(R(\alpha_j)+2\mathrm{Re}\sum\limits_{1\leq |p|\leq 2} \frac{D^p R}{p!} (\alpha_j)(\tilde z)^p 
+\frac{1}{2}\sum_{k,l=1}^n \frac{\partial^2 R}{\partial \tilde z_k\partial\overline{\tilde z_l}}(\alpha_j) \tilde z_k\overline{\tilde z_l}+o(|\tilde z|^2)\Big)\\
&=b_jR(\alpha_j)+b_j\Big(2\mathrm{Re}\sum\limits_{1\leq |p|\leq 2} \frac{D^p R}{p!}(\alpha_j) (\tilde z)^p 
+\frac{1}{2}\sum_{k,l=1}^n \frac{\partial^2 R}{\partial \tilde z_k\partial\overline{\tilde z_l}}(\alpha_j) \tilde z_k\overline{\tilde z_l}\Big)\\
&\quad+o(|\tilde z|^2)+o(|\mathrm{Im}(\tilde w)|).
\end{split}
\end{equation*}
Hence, using again Taylor's theorem we see that the hypersurface $L_{\eta_j'}(\{\rho=0\}) $ is defined by an equation of the form
\begin{equation}\label{remove-harmonic-term}
\begin{split}
\rho\left (L_{\eta_j}^{-1}(\tilde z,\tilde w)\right)&= \mathrm{Re} (\tilde w)+ R_2'(b_j) \mathrm{Im}(\tilde w) +  R(\alpha_j)\mathrm{Im}(\tilde w) +o(|\mathrm{Im}(\tilde w)|)\\
&+2\mathrm{Re}\sum\limits_{1\leq |p|\leq 2} \frac{D^pP}{p!}(\alpha_j) (\tilde z)^p+\frac{1}{2}\sum_{k,l=1}^n \frac{\partial^2 P}{\partial \tilde z_k\partial\overline{\tilde z_l}}(\alpha_j) \tilde z_k\overline{\tilde z_l}\\
&+2\mathrm{Re}\sum\limits_{1\leq |p|\leq 2} \frac{D^p R_1}{p!}(\alpha_j)  (\tilde z)^p
+\frac{1}{2}\sum_{k,l=1}^n \frac{\partial^2 R_1}{\partial \tilde z_k\partial\overline{\tilde z_l}}(\alpha_j) \tilde z_k\overline{\tilde z_l}\\
&+  b_j \Big(2\mathrm{Re}\sum\limits_{1\leq |p|\leq 2} \frac{D^p R}{p!} (\alpha_j)(\tilde z)^p 
+\frac{1}{2}\sum_{k,l=1}^n \frac{\partial^2 R}{\partial \tilde z_k\partial\overline{\tilde z_l}}(\alpha_j) \tilde z_k\overline{\tilde z_l}\Big)+o(|\tilde z|^2)=0.
\end{split}
\end{equation}

Next, to remove the pluriharmonic terms in (\ref{remove-harmonic-term}) let us define a sequence $\{Q_j\}$ of  automorphisms of $\mathbb C^{n+1}$ by
\[
\begin{cases}
w:= \tilde{w}+(R_2'(b_j) + R(\alpha_j)) i\tilde w+2\sum\limits_{1\leq |p|\leq 2} \frac{D^pP(}{p!} (\alpha_j)(\tilde z)^p+2\sum\limits_{1\leq |p|\leq 2} \frac{D^p R_1}{p!}\alpha_j)  (\tilde z)^p\\
\hskip 1cm+b_j\sum\limits_{1\leq |p|\leq 2} \frac{D^p R}{p!}(\alpha_j) ;\\
z_k:=\tilde{z}_k,\, k=1,\ldots,n.
\end{cases}
\]
Then the composite $Q_j\circ L_{\eta_j'}\in \mathrm{Aut}(\mathbb C^n)$ and satisfies that 
$$
Q_j\circ L_{\eta_j'}(\alpha_j,\beta_j)=\left(0,\ldots,0,-\epsilon_j-i(R_2'(b_j) + R(\alpha_j)) \epsilon_j\right)
$$
 for every $j\in \mathbb N^*$. Moreover, the hypersurface $Q_j\circ L_{\eta_j'}(\{\rho=0\}) $ is given by an equation of the form
\begin{align*}
\begin{split}
\rho\left (L_{\eta_j}^{-1}\circ Q_j^{-1}( z,w)\right)&= \mathrm{Re} ( w)+o(|\mathrm{Im}(w)|)+ \frac{1}{2}\sum_{k,l=1}^n \frac{\partial^2 P}{\partial z_k\partial \bar z_l}(\alpha_j) z_k\bar z_l \\
&+ \frac{1}{2}\sum_{k,l=1}^n \frac{\partial^2 R_1}{\partial z_k\partial \bar z_l}(\alpha_j) z_k\bar z_l+ \frac{b_j}{2} \sum_{k,l=1}^n \frac{\partial^2 R}{\partial z_k\partial \bar z_l}(\alpha_j) z_k\bar z_l+o(|z|^2)=0.
\end{split}
\end{align*}

Finally, let us recall the following notation
\begin{equation*}
\tau_{jk}:=|\alpha_{jk}|.\left(\dfrac{\epsilon_j}{|\alpha_{jk}|^{2m_k}}\right)^{1/2},\; 1\leq k\leq n
\end{equation*}
 and we define an anisotropic dilation $\Delta_j\colon \mathbb C^n\to \mathbb C^n$ by settings:
\begin{equation}\label{dilationj}
\Delta_j(z,w):=\Delta_{\eta_j}^{\epsilon_j} (z_1,\ldots,z_n, w)=\left(\frac{z_1}{\tau_{j1}},\ldots,\frac{z_n}{\tau_{jn}}, \frac{w}{\epsilon_j}\right).
\end{equation}
Then it follows that $\Delta_j\circ Q_j\circ L_{\eta_j'}(\alpha_j,\beta_j)=(0,\ldots,0,-1-i(R_2'(b_j) + R(\alpha_j)))\to (0',-1)$ as $j\to\infty$. Furthermore, the hypersurface $\Delta_j\circ Q_j\circ L_{\eta_j'}(\{\rho=0\}) $ is now defined by an equation of the form
\begin{align}\label{taylor-defining-function}
\begin{split}
&\epsilon_j^{-1}\rho\left (L_{\eta_j}^{-1}\circ Q_j^{-1}\circ \left(  \Delta_j \right )^{-1}(\tilde z,\tilde w)\right)\\
&= \mathrm{Re} (\tilde w)+\epsilon_j^{-1}o(\epsilon_j|\mathrm{Im}(\tilde w)|)+\frac{1}{2}\sum_{k,l=1}^n \frac{\partial^2 P}{\partial \tilde z_k\partial\overline{\tilde z_l}}(\alpha_j) \epsilon_j^{-1}\tau_{jk}\tau_{jl} \tilde z_k\overline{\tilde z_l}\\
&+\frac{1}{2}\sum_{k,l=1}^n \frac{\partial^2 R_1}{\partial \tilde z_k\partial\overline{\tilde z_l}}(\alpha_j) \epsilon_j^{-1}\tau_{jk}\tau_{jl} \tilde z_k\overline{\tilde z_l}+ \frac{\epsilon_j^{-1}b_j}{2}\sum_{k,l=1}^n \frac{\partial^2 R}{\partial \tilde z_k\partial\overline{\tilde z_l}}(\alpha_j) \tau_{jk}\tau_{jl} \tilde z_k\overline{\tilde z_l}+\cdots=0,
\end{split}
\end{align}
where the dots denote remainder terms. Note that by Lemma \ref{higher-weight} the terms with weight order greater than one must converge uniformly on compacta of  $\mathbb C^{n+1}$ to $0$. Hence, we consider only the convergence of monomials from (\ref{taylor-defining-function}) with weight order $\leq 1$.

Since the sequence $\{\eta_j:=\varphi_j(a)\}$ converges uniformly $\Lambda$-tangentially to  $\xi_0=(0',0)$, it follows that 
 $$
 \dfrac{|\alpha_{j1}|^{2m_1}}{\epsilon_j}\approx \dfrac{|\alpha_{j2}|^{2m_2}}{\epsilon_j}\approx\cdots\approx \dfrac{|\alpha_{jn}|^{2m_n}}{\epsilon_j}.
 $$
Thus Lemma \ref{Cn-spherical-convergence} yields
$$
\epsilon_j^{-1}\left |\frac{D^{p} \overline{D}^q P}{p!q!}(\alpha_j)\tau_j^{p+q}\right| \to 0
$$
as $j\to \infty$ for $|p|+|q|>2$ and, after taking a subsequence if necessary, we may assume that
$$
a_{kl}:=\frac{1}{2}\lim_{j\to\infty}\frac{\partial^2 P}{\partial \tilde z_k\partial\overline{\tilde z_l}}(\alpha_j) \epsilon_j^{-1}\tau_{jk}\tau_{jl}, 1\leq k,l\leq n.
$$
Moreover, by Lemma \ref{Cn-spherical-convergence-higher-order} we also have
$$
\epsilon_j^{-1}\left |\frac{D^{p} \overline{D}^q R_1}{p!q!}(\alpha_j)\tau_j^{p+q}\right| \to 0
$$
as $j\to \infty$ for $|p|+|q|\geq 2$. In addition, since $|\epsilon_j^{-1}b_j|\lesssim 1$, we obtain that
$$
\epsilon_j^{-1}b_j\left |\frac{D^{p} \overline{D}^q R}{p!q!}(\alpha_j)\tau_j^{p+q}\right| \to 0
$$
as $j\to\infty$ for $|p|+|q|\geq 1$. Therefore, after taking a subsequence if necessary, we may assume that sequence of defining funtions given in (\ref{taylor-defining-function}) converges uniformly on compacta of $\mathbb C^{n+1}$ to $\hat\rho(\tilde z,\tilde w):=\mathrm{Re}(\tilde w)+H(\tilde z)$. Consequently,  the sequence of domains $\Omega_j:=\Delta_j\circ Q_j\circ L_{\eta_j'}(U_0^-) $ converges normally to the following model
$$
M_{H}:=\left \{(\tilde z,\tilde w)\in \mathbb C^n\times\mathbb C\colon \hat\rho(\tilde z,\tilde w):=\mathrm{Re}(\tilde w)+H(\tilde z)<0\right\},
$$
where 
$$
H(\tilde z)=\sum_{k,l=1}^n a_{kl}  \tilde z_k\overline{\tilde z_l}.
$$

Note that $M_H$ is a limit of a sequence of the pseudoconvex domains $\Delta_j\circ Q_j\circ L_{\eta_j'}(U_0^-) $. Hence, $M_H$ is also pseudoconvex, and thus $H$ is plurisubharmonic. In addition, it follows directly from Lemma \ref{Cn-spherical-convergence} that $H$ is positive definite. Therefore, $M_{H}$ is biholomorphically equivalent to unit ball $\mathbb B^{n+1}$ (cf. \cite[Prop. 2]{Gra75}).

For simplicity, let us denote by $T_j:=\Delta_j\circ Q_j\circ L_{\eta_j'}$ and $\sigma_j:=T_j\circ \varphi_j\colon \varphi_j^{-1}(U_0^-)\to \Omega_j$. Then $T_j(\eta_j)=(0',-1-i(R_2'(b_j) + R(\alpha_j)))$ 
and $\{\sigma_j\}$ is a sequence of biholomorphic mappings satisfying 
$$
\sigma_j(a)=b_j:=(0',-1-i(R_2'(b_j) + R(\alpha_j)))\to b:= (0',-1).
$$
as $j\to\infty$. Thus, by Proposition \ref{pro-scaling}, after passing to a subsequence, we may assume that $\sigma_j$ converges locally uniformly to a holomorphic map $\sigma: \Omega \to M_H$ which satisfies $\sigma (a)=b$. 

On the other hand, since $\Omega$ is taut (cf. \cite[Prop. 2.2]{DN09}), the sequence $\sigma_j^{-1}\colon \Omega_j\to \varphi_j^{-1}(U_0^-)\subset \Omega$ is also normal. Since $\sigma_j^{-1} (b_j)=a \in \Omega$ with $b_j\to b\in \Omega$ as $j\to\infty$, we may also assume, after switching a subsequence, that
$\sigma_j^{-1}$ converges locally uniformly to a holomorphic map $\sigma^*: M_H \to \Omega$. It then follows from \cite[Prop. 2.1]{DN09} that $\Omega$ is biholomorphically equivalent to $M_H$. Hence, $\Omega$ is biholomorphically equivalent to $\mathbb B^{n+1}$, and thus the proof of Theorem \ref{maintheorem1} is finally complete. \hfill $\Box$
\subsection{Example} \label{ex3.1}  Denote by $E_{1,2,4}$ the domain in $\mathbb C^3$, given by
$$
E_{1,2,4}:=\{(z_1,z_2, w)\in \mathbb C^3\colon \mathrm{Re}(w)+|z_1|^4+|z_1|^2|z_2|^4+|z_2|^8<0\}.
$$
Denote by $P(z)=|z_1|^4+|z_1|^2|z_2|^4+|z_2|^8$ and $\sigma(z)=|z_1|^4+|z_2|^8$. Then a computation shows that 
\begin{align*}
dd^c P(z)&=(4|z_1|^2+|z_2|^4) d z_1 d\bar z_1+2 \bar z_1 z_2 |z_2|^2d z_1 d\bar z_2+ 2z_1\bar z_2|z_2|^2 d \bar z_1 d z_2\\
&+(16|z_2|^6+4|z_1|^2|z_2|^2) d z_2 d\bar z_2\\
&=4|z_1|^2 d z_1 d\bar z_1+16|z_2|^6 d z_2 d\bar z_2+ |z_2|^2 |z_2  d\bar z_1+2 \bar z_1 d z_2|^2\\
&\geq dd^c\sigma(z).
\end{align*}
Therefore, the origin is strongly $h$-extendible with multitype $(4,8,1)$ and thus the weight $\Lambda$ is now given by $\Lambda:=(\frac{1}{4},\frac{1}{8})$.

Now we consider the sequence $\{(\frac{1}{j^{1/4}},\frac{1}{j^{3/8}},-\frac{1}{j}-\frac{2}{j^2}-\frac{1}{j^3})\}$ that converges $\Lambda$-tangentially but not uniformly to $(0',0)$. We are going to show that $\{\eta_j\}$ is not a sequence of automorphism orbits, that is, there do not exist a sequence $\{f_j\}\subset \mathrm{Aut}(E_{1,2,4})$ and $a\in E_{1,2,4}$ such that $f_j(a)=(\frac{1}{j^{1/4}},\frac{1}{j^{3/8}},-\frac{1}{j}-\frac{2}{j^2}-\frac{1}{j^3})$ for all $n\in\mathbb N^*$. Assume for the sake of seeing a contradiction that $\{f_j\}$ and $a$ exist. Then, although we cannot apply a scaling given in the proof of Theorem \ref{maintheorem1}, an alternative scaling can be introduced as follows. Indeed, let $\rho(z_1,z_2,w)=\mathrm{Re}(w)+|z_1|^4+|z_1|^2|z_2|^4+|z_2|^8$ and let $\eta_j=(\frac{1}{j^{1/4}},\frac{1}{j^{3/8}},-\frac{1}{j}-\frac{2}{j^2}-\frac{1}{j^3})$ for every $j\in \mathbb N^*$. 
We see that $\rho(\eta_j)=-\frac{1}{j^2}\approx -\mathrm{dist}(\eta_j,\partial E_{1,2,4})$. Set
$\epsilon_j=|\rho(\eta_j)|=\frac{1}{j^2}$. 
In addition, we consider a change of variables $(\tilde z,\tilde w):=L_j(z,w)$, i.e.,
\[
\begin{cases}
w=\tilde w;\\
\displaystyle z_1-\frac{1}{j^{1/4}}= \tilde{z}_1;\\
\displaystyle z_2-\frac{1}{j^{3/8}}=\tilde{z}_2.
\end{cases}
\]
Then, a direct calculation shows that
\begin{equation*}
\begin{split}
&\rho\circ L_j^{-1} (\tilde w,\tilde z_1,\tilde z_2)=
\mathrm{Re}(w)+ |\dfrac{1}{j^{1/4}}+\tilde z_1|^4+ |\dfrac{1}{j^{1/4}}+\tilde z_1|^2 |\frac{1}{j^{3/8}}+\tilde z_2|^4+|\frac{1}{j^{3/8}}+\tilde z_2|^8 \\
                      &= \mathrm{Re}(w)+\frac{1}{j}+\frac{4}{j^{3/4}}\mathrm{Re}(\tilde z_1) +\frac{2}{j^{1/2}}|\tilde z_1|^2+\frac{1}{j^{1/2}}(2\text{Re}(\tilde z_1))^2
+ \frac{4}{j^{1/4}} |\tilde z_1|^2 \mathrm{Re}(\tilde z_1)+ |\tilde z_1|^4\\
&+\left( \frac{1}{j^{1/2}}+\dfrac{2}{j^{1/4}}\mathrm{Re}(\tilde z_1)+|\tilde z_1|^2\right) \times\\
 &\left(\frac{1}{j^{3/2}}+\frac{4}{j^{9/8}}\mathrm{Re}(\tilde z_2) +\frac{2}{j^{3/4}}|\tilde z_2|^2+\frac{1}{j^{3/4}}(2\text{Re}(\tilde z_2))^2
+ \frac{4}{j^{3/8}} |\tilde z_2|^2 \mathrm{Re}(\tilde z_2)+ |\tilde z_2|^4\right)+|\frac{1}{j^{3/8}}+\tilde z_2|^8 .
\end{split}
\end{equation*}

To define an anisotropic dilation,  let us denote by
 $\tau_{1j}:=\tau_1(\eta_j)=\frac{1}{2 j^{3/4}}, \; \tau_{2j}:=\tau_2(\eta_j)=\frac{1}{ j^{3/8}}$ for all $j\in \mathbb N^*$. Now let us introduce a sequence of  polynomial automorphisms $\phi_{{\eta}_j}$ of $\mathbb C^n$ ($j\in \mathbb N^*$), given by
\begin{equation*}
\begin{split}
&\phi_{{\eta}_j} ^{-1}(\tilde z_1,\tilde z_2,\tilde w)\\
&= \Big (\dfrac{1}{j^{1/4}}+\tau_{1j} \tilde z_1, \,\dfrac{1}{j^{3/8}}+ \tau_{2j} \tilde z_2, \,-\frac{1}{j}-\frac{1}{j^2}-\frac{1}{j^3}+\epsilon_j \tilde w- \frac{4}{j^{3/4}}\tau_{1j} \tilde z_1-\frac{2}{j^{1/2}}(\tau_{1j})^2 \tilde z_1^2\Big).
\end{split}
\end{equation*}
Therefore, for each $j\in\mathbb N^*$ the hypersurface $\phi_{\eta_j}(\{\rho=0\}) $ is then defined by

\begin{equation*}
\begin{split}
&\epsilon_j^{-1}\rho\circ \phi_{{\eta}_j} ^{-1}(\tilde z_1,\tilde z_2,\tilde w)\\
&= \epsilon_j^{-1}\rho \Big (\dfrac{1}{j^{1/4}}+\tau_{1j} \tilde z_1, \,\dfrac{1}{j^{3/8}}+ \tau_{2j} \tilde z_2, \,-\frac{1}{j}-\frac{1}{j^2}-\frac{1}{j^3}+\epsilon_j \tilde w- \frac{4}{j^{3/4}}\tau_{1j} \tilde z_1-\frac{2}{j^{1/2}}(\tau_{1j})^2 \tilde z_1^2\Big)\\
&=\mathrm{Re}(\tilde w) +|\tilde z_1|^2+ \frac{1}{16j}|\tilde z_1|^4
+ \frac{1}{2j^{1/4}}|\tilde z_1|^2\mathrm{Re}(\tilde z_1)+ \left(|\tilde z_2+1|^4-1\right)+O(\frac{1}{j^{1/2}})+O(\frac{1}{j})=0.
\end{split}
\end{equation*} 
Hence, the sequence of domains $\Omega_j:=\phi_{{\eta}_j}(E_{1,2,4}) $ converges normally to the following model
$$
M_{1,2}:=\left \{(\tilde z_1,\tilde  z_2, \tilde  w)\in \mathbb C^3\colon \mathrm{Re}(\tilde  w)+|\tilde  z_1|^2+ \left(|\tilde z_2+1|^4-1\right)<0\right\}.
$$

Finally, by the same argument as in the proof of Theorem \ref{maintheorem1} we conclude that $E_{1,2,2}$ is biholomorphically equivalent to $M_{1,2}$ that is biholomorphically equivalent to 
$$
\left \{(z_1,z_2, w)\in \mathbb C^3\colon \mathrm{Re}(w)+|z_1|^2+ |z_2|^4<0\right\}.
$$
It is absurd by \cite[Main Theorem]{CP01}.

\begin{remark}\label{model-behavior} We consider a sequence $\{(\frac{1}{j^{1/4}},\eta_{j2},-P(\frac{1}{j^{1/4}},\eta_{j2})-\frac{1}{j^2}\}$ that converges $\Lambda$-tangentially but not uniformly to $(0',0)$. Then, following the argument as in Example \ref{ex3.1}, we may assume that $\Omega_j:=\phi_{{\eta}_j}(E_{1,2,4}) $ converges normally to the following model
$$
\left \{(\tilde z_1,\tilde  z_2, \tilde  w)\in \mathbb C^3\colon \mathrm{Re}(\tilde  w)+|\tilde  z_1|^2+ \tilde P(\tilde z_2)<0\right\},
$$
where 
\begin{itemize}
\item[i)] $\tilde P(\tilde z_2)=|\tilde z_2+\alpha|^4-|\alpha|^4$ for some $\alpha \in \mathbb C$ if $|\eta_{j2}|\approx \frac{1}{j^{3/8}}$;
\item[ii)] $\tilde P(\tilde z_2)=|\tilde z_2|^4$ if $|\eta_{j2}|=o\left( \frac{1}{j^{3/8}}\right)$;
\item[iii)] $\tilde P(\tilde z_2)=|\tilde z_2|^2$ if $\frac{1}{j^{3/8}}=o(|\eta_{j2}|) $.
\end{itemize}
This indicates that our model depends deeply on the behavior of the orbit $\{\eta_j\}\subset \Omega$.

\end{remark}

\section{The behaviour of automorphism orbits accumulating at a boundary point of a pseudoconvex Levi corank one domains in $\mathbb C^{n+1}$}\label{S-corank1}
In this section, we are going to give a proof of Theorem \ref{maintheorem3}. To do that, let $\Omega$ be a domain in $\mathbb C^{n+1}$ such that $\partial \Omega$ is pseudoconvex of finite type and has corank one near $\xi_0$.
\subsection{Spherical $\frac{1}{2m}$-tangential convergence}
In what follows, let us write $z=(z_1,\ldots,z_n)$ and $z^*=(0,z_2,\ldots,z_{n})$. Let $2m$ be the D'Angelo type of $\partial \Omega$ at $\xi_0$. Without loss of generality, we may assume that $\xi_0$ and the rank of Levi form at $\xi_0$ is exactly $n-1$. Let $\rho$ be a smooth defining function for  $\Omega $. After an appropriate change of coordinates (cf. \cite{BP91, Cho94}), we can find coordinate functions $z_1,\ldots, z_n,w$ defined on a neighborhood $U_0$ of $\xi_0$ such that $\xi_0=0$ and
\begin{equation*}
\begin{split}
\rho(z)&=\mathrm{Re}(w)+ P(z_1,\bar z_1)+\sum_{\alpha=2}^{n}|z_\alpha|^2+ \sum_{\alpha=2}^{n} \mathrm{Re} (Q^\alpha(z_1,\bar z_1)z_\alpha)\\
&+O(|w| |(z,w)|+|z^*|^2|z|+|z^*|^2|z_1|^{m+1}+|z_1|^{2m+1}),
\end{split}
\end{equation*}
where $P(z_1,\bar z_1), Q^\alpha(z_1,\bar z_1)\; (2\leq \alpha\leq n)$ are homogeneous subharmonic real-valued polynomials of degree $2m$ and $m$, respectively, containing no harmonic terms.
 
\begin{define}\label{spherically-convergence}
We say that a sequence $\{\eta_j=(\alpha_j,\beta_j)\}\subset  \Omega$ with $\alpha_j=(\alpha_{j 1},\ldots,\alpha_{j n})$, \emph{converges spherically $\frac{1}{2m}$-tangentially to $\xi_0$} if
\begin{itemize}
\item[(a)] $|\mathrm{Im}(\beta_j)|\lesssim |\mathrm{dist}(\eta_j,\partial \Omega)|$;
\item[(b)] $|\mathrm{dist}(\eta_j,\partial \Omega)|=o(|\alpha_{j1}|^{2m})$;
\item[(c)]  $\Delta P(\alpha_{j1})\gtrsim |\alpha_{j1}|^{2m-2}$.
\end{itemize}
\end{define}
\begin{remark} Let $\Omega$ be a pseudoconvex domain in $\mathbb C^2$. Suppose that $\xi_0\in\partial \Omega$ is of D'Angelo finite type, say, $\tau(\partial\Omega,\xi_0)=2m$. It is known that $\Omega$ is $h$-extendible at $\xi_0$. Let $\{\epsilon_j\}\subset \mathbb R^+$ be a sequence  such that $\eta_j':=(\alpha_j,\beta_j+\epsilon_j)$ is in the hypersurface $\{\rho=0\}$ for every $j\in\mathbb N^*$. Then the condition $\mathrm{(c)}$ simply says that $\Omega$ is strongly pseudoconvex at $\eta_j'$ for every $j\in\mathbb N^*$. Consequently, the condition $\mathrm{(c)}$ is clearly satisfied if  the model $M_P:=\left\{(z,w)\in\mathbb C^2\colon \mathrm{Re}(w)+P(z_1)<0\right\}$ is a WB-domain.
\end{remark}

\subsection{Homogeneous subharmonic polynomials}
Let us write $\displaystyle P(z)=\sum_{j=1}^{2m-1} a_j z^j\bar z^{2m-j}$. Writing $z=|z| e^{i\theta}$, one defines $g(\theta)$ by
 $$
 P(z)=|z|^{2m} g(\theta).
 $$
 Then we have
 $$
 \Delta P(z)=|z|^{2m-2} \left((2m)^{2} g(\theta)+g_{\theta\theta}(\theta)\right)\geq 0.
 $$
 (See cf. \cite{BF78}.)
 
  Given a sequence $\{\epsilon_j\}\subset \mathbb R^+$, we associate the sequence $\{\tau_j\}$ given by
\begin{equation*}
\tau_{j}:=\tau(\alpha_j,\epsilon_j)=|\alpha_{j}|.\left(\dfrac{\epsilon_j}{|\alpha_{j}|^{2m}}\right)^{1/2}.
\end{equation*}
 To give a proof of Theorem \ref{maintheorem3}, it suffices to check the following lemma that is similar to Lemma \ref{Cn-spherical-convergence}.
 \begin{lemma}\label{spherical-convergence}
 We have
$$
\left |\frac{\partial^k P}{\partial z^l\partial \bar z^{k-l}}(\alpha_j)\right| \epsilon_j^{-1}\tau_j^{k}\lesssim \left(\dfrac{|\alpha_j|^{2m}}{\epsilon_j}\right)^{1-\frac{k}{2}},\; k\geq 3.
$$
In addition, if $k=2$ and $j=1$, then 
$$
\left |\frac{\partial^2 P}{\partial z\partial \bar z}(\alpha_j)\right| \epsilon_j^{-1}\tau_j^{2}=(2m)^2g(\theta_j)+g_{\theta\theta}(\theta_j).
$$
\end{lemma}
\begin{proof} It suffices to prove the lemma for $P(z)=z^s \bar z^{2m-s}$ for $1\leq s\leq 2m-1$. Indeed, a simple computation shows that
\[
\frac{\partial^k P}{\partial z^l\partial \bar z^{k-l}}(z)=
\begin{cases} z^{s-l}\bar z^{2m-s-(k-l)}\; &\text{if } l\leq s, k-l\leq 2m-s\\
0 \; &\text{if otherwise}
\end{cases}
\]
and
\begin{align*}
\left |\frac{\partial^k P}{\partial z^l\partial \bar z^{k-l}}(\alpha_j)\right| \epsilon_j^{-1}\tau_j^{k}&\lesssim |\alpha_j|^{2m-k} \epsilon_j^{-1}\tau_j^{k}=|\alpha_j|^{2m} \epsilon_j^{-1}\left(\frac{\tau_j}{|\alpha_j|}\right)^{k}\\
&\lesssim \frac{|\alpha_j|^{2m}}{ \epsilon_j}\left(\frac{\epsilon_j}{|\alpha_j|^{2m}}\right)^{k/2}=\left(\frac{\epsilon_j}{|\alpha_j|^{2m}}\right)^{k/2-1}.
\end{align*}
Therefore, we get 
$$
\left |\frac{\partial^k P}{\partial z^l\partial \bar z^{k-l}}(\alpha_j)\right| \epsilon_j^{-1}\tau_j^{k}\lesssim \left(\dfrac{|\alpha_j|^{2m}}{\epsilon_j}\right)^{1-\frac{k}{2}} 
$$
for $k>2$. Furthermore, in the case that $k=2, l=1$ one has
\begin{align*}
\left |\frac{\partial^2 P}{\partial z\partial \bar z}(\alpha_j)\right| \epsilon_j^{-1}\tau_j^{2}&=|\alpha_j|^{2m-2}\left((2m)^2g(\theta_j)+g_{\theta\theta}(\theta_j)\right)|\alpha_j|^2 \left(\frac{\epsilon_j}{|\alpha_j|^{2m}}\right)\\
&=(2m)^2g(\theta_j)+g_{\theta\theta}(\theta_j).
\end{align*}
\end{proof}

\subsection{Proof of Theorem \ref{maintheorem3}}
Throughout this section, the domain $\Omega$ and the boundary point $\xi_0\in \partial \Omega $ are assumed to satisfy the hypothesis of Theorem \ref{maintheorem3}. Let $2m$ be the D'Angelo type of $\partial \Omega$ at $\xi_0$. Without loss of generality, we may assume that $\xi_0=0\in\mathbb C^n$ and the rank of Levi form at $\xi_0$ is exactly $n-1$. Let $\rho$ be a smooth defining function for  $\Omega $. After an appropriate change of coordinates (cf. \cite{BP91, Cho94}), we can find the coordinate functions $z_1,\ldots, z_n,w$ defined on a neighborhood $U_0$ of $\xi_0$ such that $\xi_0=0$ and
\begin{equation*}
\begin{split}
\rho(z,w)&=\mathrm{Re}(w)+ P(z_1,\bar z_1)+\sum_{\alpha=2}^{n}|z_\alpha|^2+ \sum_{\alpha=2}^{n} \mathrm{Re} (Q^\alpha(z_1,\bar z_1)z_\alpha)\\
&+O(|w| |(z,w)|+|z^*|^2|z|+|z^*|^2|z_1|^{m+1}+|z_1|^{2m+1}),
\end{split}
\end{equation*}
where $P(z_1,\bar z_1), Q^\alpha(z_1,\bar z_1)\; (2\leq \alpha\leq n)$ are homogeneous subharmonic real-valued polynomials of degree $2m$ and $m$, respectively, containing no harmonic terms.

By hypothesis of Theorem \ref{maintheorem3}, there exist a sequence $\{\varphi_j\}\subset \mathrm{Aut}(\Omega)$ and a point $a\in \Omega$ such that $\eta_j:=\varphi_j(a)$ converges spherically $\frac{1}{2m}$-tangentially to $\xi_0$. Let us write $\eta_j=(\alpha_j,\beta_j)=(\alpha_{j},\beta_j)$. Then one has 
\begin{itemize}
\item[(a)] $|\mathrm{Im}(\beta_j)|\lesssim |\mathrm{dist}(\eta_j,\partial \Omega)|$;
\item[(b)] $|\mathrm{dist}(\eta_{j},\partial \Omega)|=o(|\alpha_{j1}|^{2m})$;
\item[(c)] $\Delta P(\alpha_{j1})\gtrsim |\alpha_{j1}|^{2m-2}$.
\end{itemize}

Let us fix a small neighborhood $U_0$ of  the origin. For any sequence $\{\eta_j=(\alpha_j,\beta_j)\}$ of points converging $\frac{1}{2m}$-tangentially to the origin in $U_0\cap\{\rho<0\}=:U_0^-$, we associate with a sequence of points $\eta_j'=(\alpha_{j}, a_j +\epsilon_j+i b_j)$, where $\epsilon_j>0$ and $\beta_j=a_j+i b_j$, such that $\eta_j'=(\alpha_j,\beta_j')$ is in the hypersurface $\{\rho=0\}$ for every $j\in\mathbb N^*$. We note that $\epsilon_j\approx \mathrm{dist}(\eta_j,\partial \Omega)$.

By \cite[Proposition~$2.2$]{Cho94} (see also~\cite[Proposition~$3.1$]{DN09}),  for each point $\eta_j'$, there exists a biholomorphism $\Phi_{\eta_j'}$ of $\mathbb C^{n+1}$, $(z,w)=\Phi^{-1}_{\eta_j'}(\tilde z, \tilde w)$, such that
\begin{equation}\label{Eq19} 
\begin{split}
\rho(\Phi_{\eta_j'}^{-1}(\tilde z,\tilde w))&= \mathrm{Re}(\tilde  w)+ \sum_{\substack{k+l\leq 2m\\
 k,l>0}} a_{k,l}(\eta_j')\tilde z_1^k \overline{\tilde z_1}^l\\
&+\sum_{\alpha=2}^{n-1}|\tilde z_\alpha|^2+ \sum_{\alpha=2}^{n} \sum_{\substack{k+l\leq m\\
 k,l>0}}\mathrm{Re} [(b^\alpha_{k,l}(\eta_j'))\tilde  z_1^k \overline{\bar z_1}^l)\tilde z_\alpha]\\
&+O(|\tilde w| |(\tilde z, \tilde w)|+|\tilde  z^*|^2|\tilde w|+|\tilde  z^*|^2|\tilde z_1|^{m+1}+|\tilde z_1|^{2m+1}),
\end{split}
\end{equation}
where $\tilde z^*=(0,\tilde z_2,\ldots,\tilde z_{n})$.

Now let us define 
\begin{equation*}
\tau_{j1}:=|\alpha_{j}|.\left(\dfrac{\epsilon_j}{|\alpha_{j}|^{2m}}\right)^{1/2}, \tau_{j2}=\cdots=\tau_{jn}=\epsilon_j^{1/2},\; j\geq 1.
\end{equation*}
This implies that 
$$
\epsilon_j^{1/2}\lesssim \tau(\eta_j,\epsilon_j)  \lesssim \epsilon_j^{1/(2m)}.
$$

To finish the scaling procedure,  we define an anisotropic dilation $\Delta_j$ by 
\[
\Delta_j (z,w)=\left( \frac{z_1}{\tau_{j1}},\frac{z_2}{\tau_{j2}}, \ldots,\frac{z_n}{\tau_{jn}},\frac{w}{\epsilon_j}\right),\; j\in \mathbb N^*.
\]
This yields $\Delta_j\circ \Phi_{\eta_j'}(\eta_j)=(0',-1+\gamma_j)$ for some sequence $\{\gamma_j\}\subset \mathbb C$, depending on $\{ \Phi_{\eta_j'}\}$, that converges to $0$ as $j\to\infty$. Furthermore, for each $j\in \mathbb N^*$, if we set $\rho_j(z,w)=\epsilon_j^{-1}\rho\circ \Phi_{\eta_j'}^{-1}\circ(\Delta_j)^{-1}(z,w)$, then \eqref{Eq19} implies that
\[
\rho_j(z,w)=\mathrm{Re}(w)+ P_{\eta_j'}(z_1,\bar z_1)+\sum_{\alpha=2}^{n}|z_\alpha|^2+ \sum_{\alpha=2}^{n}\mathrm{Re}(Q^\alpha_{\eta_j'}(z_1,\bar z_1)z_\alpha)+O(\tau(\eta_j',\epsilon_j)),
\]
where
\begin{equation*}
\begin{split}
&P_{\eta_j'}(z_1,\bar z_1):=\sum_{\substack{k,l\leq 2m\\
 k,l>0}} a_{k,l}(\eta_j') \epsilon_j^{-1} \tau(\eta_j',\epsilon_j)^{k+l}z_1^k \bar z_1^l,\\
&Q^\alpha_{\eta_j'}(z_1,\bar z_1):= \sum_{\substack{k+l\leq m\\
 k,l>0}} b^\alpha_{k,l}(\eta_j')\epsilon_j^{-1/2} \tau(\eta_j',\epsilon_j)^{k+l}z_1^k \bar z_1^l.
\end{split}
\end{equation*}

Note that the sequence $\{\eta_j:=\varphi_j(a)\}$ converges spherically $\frac{1}{2m}$-tangentially to  $\xi_0=(0',0)$. Then $\dfrac{|\alpha_{j1}|^{2m}}{\epsilon_j}\to +\infty$ as $j\to\infty$. In addition, we have, for $1\leq k,l$ with $k+l\leq 2m$, that
$$
 a_{k,l}(\eta_j')=\frac{\partial^{k+l} \rho}{\partial \tilde z_1^k\partial \overline{\tilde z_1}^{l}}(0',0)\approx  \frac{\partial^{k+l} P}{\partial z_1^k\partial \bar z_1^{l}}(\alpha_{j1}).
$$
Therefore, by Lemma \ref{spherical-convergence} we get $a_{k,l}(\eta_j') \epsilon_j^{-1} \tau(\eta_j',\epsilon_j)^{k+l}\to 0$ as $j\to \infty$ for $k,l>0$ with $2<k+l\leq 2m$. In addition, by the condition $\mathrm{(c)}$, without loss of generality, we may assume that the limit $\displaystyle  a:=\lim_{j\to \infty}\frac{1}{2}  \frac{\partial^2 P}{\partial z_1\partial \bar z_1}(\alpha_{j1})\epsilon_j^{-1} \tau_{j1}^2 > 0$ exists, and thus we conclude that $\{P_{\eta_j'}(z_1,\bar z_1)\}$ converges uniformly on compacta to $a|z_1|^2$.

For the sequences $\{Q^\alpha_{\eta_j'}(z_1,\bar z_1)\}$ ($2\leq \alpha\leq n$),  by \cite[Lemma $2.4$]{Cho94} it follows that 
$$
|Q^\alpha_{\eta_j'}(z_1,\bar z_1)|\leq \tau(\eta_j',\epsilon_j)^{\frac{1}{10}},\, j\geq 1,
$$
for all $\alpha=2,\ldots, n$ and $|z_1|\leq 1$. Consequently, $\{Q^\alpha_{\eta_j'}\}$ converge uniformly on every compact subset of $\mathbb C$ to $0$. Therefore, after taking a subsequence if necessary, we may assume that the sequence $\{\hat \rho_j\}$ converges to the following function
$$
\hat\rho(z,w):=\mathrm{Re}(w)+a|z_1|^2+|z_2|^2+\cdots+|z_n|^2,
$$
where $\displaystyle a=\frac{1}{2}  \frac{\partial^2 P}{\partial z_1\partial \bar z_1}(\alpha_{j1})\epsilon_j^{-1} \tau_{j1}^2 > 0$. Hence, after taking a subsequence if necessary, we may assume that the sequence of domains $\Omega_j:=\Delta_j\circ \Phi_{\eta_j'}(U_0^-) $ converges normally to the Siegel half-space
$$
M_{|z|^2}:=\left \{( z,w)\in \mathbb C^{n+1}\colon \hat\rho(z,w)=\mathrm{Re}(w)+a|z_1|^2+|z_2|^2+\cdots+|z_n|^2<0\right\},
$$
which is clearly biholomorphically equivalent to the unit ball $\mathbb B^{n+1}$ (by using the Cayley transform). Moreover, by the same argument as in the proof of Theorem \ref{maintheorem1} we conclude that $\Omega$ is biholomorphically equivalent to $M_{|z|^2}$, or $\Omega$ is biholomorphically  equivalent to $\mathbb B^{n+1}$. Hence, the proof of Theorem \ref{maintheorem3} is finally complete. \hfill $\Box$

\section{The behaviour of automorphism orbits accumulating at a boundary point of a domain in $\mathbb C^2$}\label{Sec5}

In this section, we shall restrict the discussion to domains in $\mathbb C^2$. More precisely, let $\Omega$ be a pseudoconvex domain of finite type near $\xi_0\in \partial\Omega$ with the type $\tau(\partial \Omega,\xi_0)=2m$. Then the notion of uniformly $(\frac{1}{2m})$-tangential convergence given in Section \ref{S-h-extendible} is just the that of $(\frac{1}{2m})$-tangential convergence. In addition, the notion of spherically $\frac{1}{2m}$-tangential convergence is exactly given in Section \ref{S-corank1}. Therefore, Corollary \ref{maintheorem2} follows directly from Theorem \ref{maintheorem3} and Lemma \ref{spherical-convergence}. In this situation our model is biholomorphically equivalent to the unit ball $\mathbb B^2$. 

In the sequel, we consider the case that the condition $c)$ in Definition \ref{spherically-convergence} is violated, i.e. $\dfrac{\Delta P(\alpha_j)}{|\alpha_j|^{2m-2}}\to 0$ as $j\to \infty$ for some sequence $\{\alpha_j\}$ converging to the origin in $\mathbb C$ (see Definition \ref{spherically-convergence-2} below). Then our model may be defined by a homogeneous polynomial of degree larger than $2$.
\subsection{Spherically $\frac{1}{2m}$-tangentially convergence of higher order }
Let $\rho$ be a local defining function for $\Omega$ near $\xi_0$ and let the D'Angelo type $\tau(\partial \Omega,\xi_0)=2m$ be finite. As in the proof of Theorem \ref{maintheorem1}, we may assume that there are local holomorphic coordinates $(z,w)$ in which $\xi_0=0$ and $\Omega$ can be described near $0$ as follows: 
$$
 \Omega=\left\{\rho(z,w)=\mathrm{Re}(w)+ P(z)+R_1(z)+ R_2(\mathrm{Im} w)+(\mathrm{Im} w) R(z)<0\right\}.
 $$ 
 Here $P$ is homogeneous subharmonic real-valued polynomial of degree $2m$ containing no harmonic terms, $R_1\in \mathcal{O}(1, \Lambda),R\in \mathcal{O}(1/2, \Lambda) $ with $\Lambda=(\frac{1}{2m})$, and $R_2\in \mathcal{O}(2)$.  

Now let us write 
$$
P(z)= \sum_{l=1}^{2m-1} a_l z^l \bar z^{2m-l},
$$
where $a_l=\overline{a_{l'}}$ if $l+l'=2m$. Furthermore, in the polar coordinates $z=|z| e^{i \theta},$ for $l+l' \leq 2m$, we have
\begin{equation} \label{derivatives}
\frac{\partial^{l+l'} P}{\partial z^l \partial \bar z^{l'}} (z)= |z|^{2m-l-l'} g_{l,l'} (\theta);\;\frac{\partial^{l+l'} R_1}{\partial z^l \partial \bar z^{l'}} (z)= |z|^{2m-l-l'} h_{l,l'} (|z|,\theta),
\end{equation}
where $g_{l,l'}$ is a trigonometric polynomial of degree $2m-l-l'$ and  $h_{l,l'} (|z|,\theta)$ is a $\mathcal{C}^\infty$-smooth function defined on $\mathbb C$ with $h_{l,l'} (|z|,\theta)=O(|z|)$.

In analogy to Corollary \ref{maintheorem2}, we discuss a situation where our domain is biholomophically equivalent to some model defined by a homogeneous polynomial of degree larger than $2.$ To this purpose, the following variant of Definition \ref{spherically-convergence} seems to be necessary.
\begin{define}\label{spherically-convergence-2}
A sequence $\{\eta_j=(\alpha_j,\beta_j)\} \subset \Omega$ is said to converge spherically $\frac{1}{2m}$-tangentially of order $2\nu\ (2 \leq \nu \leq m)$ to $\xi_0 \in \partial \Omega$ if the following conditions hold:
\begin{itemize}
\item[(i)] $\vert \mathrm{Im} (\beta_j) \vert \leq  \mathrm{dist} (\eta_j, \partial \Omega);$
\item[(ii)] $\mathrm{dist} (\eta_j, \partial \Omega) =o(\vert \alpha_j\vert^{2m})$;
\item[(iii)] If $l+l'<2\nu$, then 
$$\lim_{j \to \infty} \Big (\frac{\mathrm{dist} (\eta_j, \partial \Omega)}{\vert \alpha_j\vert^{2m}} \Big)^{\frac{l+l'}{2\nu}-1} \left(g_{l,l'} (\theta_j)+ h_{l,l'} (|\alpha_j|,\theta_j)\right)=0,$$
where $\theta_j:=\arg (\alpha_j);$
\item[(iv)]There exists $l_0, l_0'$ with $l_0+l_0' =2\nu, \max (l_0, l'_0) \geq 1$ such that 
$$\liminf\limits_{j \to \infty}| g_{l_0, l_0'} (\theta_j)|>0.$$ 
\end{itemize}
\end{define}

We now assume that $\{\eta_j=(\alpha_j,\beta_j)\} \subset \Omega$ converges spherically $\frac{1}{2m}$-tangentially of order $2\nu\ (2 \leq \nu \leq m)$ to $\xi_0 \in \partial \Omega$. Then, for a given sequence $\{\epsilon_j\}\subset \mathbb R^+$ converging to $0$, we define 
$$
\tau_j:= \vert \alpha_j\vert \Big (\frac{\epsilon_j}{\vert \alpha_j\vert^{2m}} \Big)^{\frac{1}{2\nu}},\; j\geq 1.
$$
With this notation, by Definition \ref{spherically-convergence-2} we have the following lemma.
\begin{lemma}\label{higher-order-model} For any integers $l,l'$ with $l,l'\geq 1$, we have
\begin{itemize}
\item[(a)]
$\displaystyle \lim_{j \to \infty}\frac{\partial^{l+l'} (P+R_1)}{\partial z^l \partial \bar z^{l'}} (\alpha_j) \epsilon_j^{-1} \tau_j^{l+l'}=0
$
for $l+l'<2\nu$.
\item[(b)]
$\displaystyle \lim_{j \to \infty}\frac{\partial^{l+l'} P}{\partial z^l \partial \bar z^{l'}} (\alpha_j) \epsilon_j^{-1} \tau_j^{l+l'}=0
$
for $l+l'>2\nu$.
\item[(c)]
$\displaystyle\lim_{j \to \infty}\frac{\partial^{l+l'} R_1}{\partial z^l \partial \bar z^{l'}} (\alpha_j) \epsilon_j^{-1} \tau_j^{l+l'}=0
$
for $l+l'\geq 2\nu$.
\item[(d)] $\displaystyle\left|\frac{\partial^{l+l'} P}{\partial z^l \partial \bar z^{l'}} (\alpha_j)\right| \epsilon_j^{-1} \tau_j^{l+l'}\lesssim 1$ for $l+l'= 2\nu$ and 
$$\liminf\limits_{j \to \infty}\left|\frac{\partial^{2\nu} P}{\partial z^{l_0} \partial \bar z^{2\nu-l_0}} (\alpha_j)\right| \epsilon_j^{-1} \tau_j^{2\nu}=\liminf\limits_{j \to \infty} g_{l_0, 2\nu-l_0} (\theta_j)>0.
$$
\end{itemize}
\end{lemma}
\begin{proof}
Indeed, by a direct computation using (\ref{derivatives}) we obtain
$$
\begin{aligned} 
\frac{\partial^{l+l'} P}{\partial z^l \partial \bar z^{l'}} (\alpha_j) \epsilon_j^{-1} \tau_j^{l+l'}
&=\vert \alpha_j\vert^{2m-l-l'} g_{l,l'} (\theta_j) \epsilon_j^{-1} \vert \alpha_j\vert^{l+l'} \Big (\frac{\epsilon_j}{\vert \alpha_j\vert^{2m}} \Big)^{\frac{l+l'}{2\nu}}\\
&=\Big (\frac{\epsilon_j}{\vert \alpha_j\vert^{2m}} \Big)^{\frac{l+l'}{2\nu}-1} g_{l,l'} (\theta_j)
\end{aligned}
$$
for $l,l'\geq 1$. Similarly, one also has
$$ 
\frac{\partial^{l+l'} R_1}{\partial z^l \partial \bar z^{l'}} (\alpha_j) \epsilon_j^{-1} \tau_j^{l+l'}=\Big (\frac{\epsilon_j}{\vert \alpha_j\vert^{2m}} \Big)^{\frac{l+l'}{2\nu}-1} h_{l,l'} (|\alpha_j|,\theta_j)
$$
for $l,l'\geq 1$. Thus the assertion (a) follows directly from the condition (iii). 

In the same fashion, for $l+l' \geq 2\nu$ we have
\begin{equation} \label{eq-15-2}
\begin{split}
\left|\frac{\partial^{l+l'} P}{\partial z^l \partial \bar z^{l'}} (\alpha_j)\right| \epsilon_j^{-1} \tau_j^{l+l'}
& \lesssim \vert \alpha_j\vert^{2m-l-l'} \epsilon_j^{-1} \vert \alpha_j\vert^{l+l'} \Big (\frac{\epsilon_j}{\vert \alpha_j\vert^{2m}} \Big)^{\frac{l+l'}{2\nu}}\lesssim\Big (\frac{\epsilon_j}{\vert \alpha_j\vert^{2m}} \Big)^{\frac{l+l'}{2\nu}-1};\\
\left|\frac{\partial^{l+l'} R_1}{\partial z^l \partial \bar z^{l'}} (\alpha_j)\right| \epsilon_j^{-1} \tau_j^{l+l'}
& \lesssim \vert \alpha_j\vert^{2m+1-l-l'} \epsilon_j^{-1} \vert \alpha_j\vert^{l+l'} \Big (\frac{\epsilon_j}{\vert \alpha_j\vert^{2m}} \Big)^{\frac{l+l'}{2\nu}}\lesssim|\alpha_j| \Big (\frac{\epsilon_j}{\vert \alpha_j\vert^{2m}} \Big)^{\frac{l+l'}{2\nu}-1}.
\end{split}
\end{equation}
This easily implies $(c)$. Moreover, if $l+l'>2\nu$, then we get
\begin{equation*}
\lim_{j \to \infty}\frac{\partial^{l+l'} P}{\partial z^l \partial \bar z^{l'}} (\alpha_j) \epsilon_j^{-1} \tau_j^{l+l'}=0,
\end{equation*}
and hence (b) follows.

Finally, it follows from (\ref{eq-15-2}) and the condition (iii) that we observe that every sequence $\Big \{\frac{\partial^{2\nu} P}{\partial z^l \partial \bar z^{l'}} (\alpha_j) \epsilon_j^{-1} \tau_j^{l+l'}\Big\}_{j \ge 1}$ is bounded if $l+l'=2\nu,$  whereas
$$\liminf\limits_{j \to \infty}\left|\frac{\partial^{2\nu} P}{\partial z^{l_0} \partial \bar z^{2\nu-l_0}} (\alpha_j)\right| \epsilon_j^{-1} \tau_j^{2\nu}=\liminf\limits_{j \to \infty} g_{l_0, 2\nu-l_0} (\theta_j)>0.
$$ 
\end{proof}
\subsection{Proof of Proposition \ref{1.2new}}
This subsection is devoted to a proof of Proposition \ref{1.2new}. Throughout this section, the domain $\Omega$ and the boundary point $\xi_0\in \partial \Omega $ are assumed to satisfy the hypothesis of Proposition \ref{1.2new}. Let $\rho$ be a local defining function for $\Omega$ near $\xi_0$ and let the D'Angelo type $\tau(\partial \Omega,p)=2m$ be finite. As in the proof of Theorem \ref{maintheorem1}, we may assume that there are local holomorphic coordinates $(z,w)$ in which $\xi_0=0$ and $\Omega$ can be described near $0$ as follows: 
\begin{align*}
 \Omega=\left\{\rho(z,w)=\mathrm{Re}(w)+ P(z)+R_1(z)+ R_2(\mathrm{Im} w)+(\mathrm{Im} w) R(z)<0\right\}.
\end{align*} 
 Here $P$ is homogeneous subharmonic real-valued polynomial of degree $2m$ containing no harmonic terms, $R_1\in \mathcal{O}(1, \Lambda),R\in \mathcal{O}(1/2, \Lambda) $ with $\Lambda=(\frac{1}{2m})$, and $R_2\in \mathcal{O}(2)$.  

By hypothesis of Proposition \ref{1.2new}, there exist a sequence $\{\varphi_j\}\subset \mathrm{Aut}(\Omega)$ and a point $a\in \Omega$ such that 
$\eta_j:=\varphi_j(a)$ converges spherically $\frac{1}{2m}$-tangentially of order $2\nu\ (2 \leq k \leq m)$ to $\xi_0$. 

Let us fix a small neighborhood $U_0$ of  the origin. For any sequence $\{\eta_j=(\alpha_j,\beta_j)\}$ of points converging $\frac{1}{2m}$-tangentially to the origin in $U_0\cap\{\rho<0\}=:U_0^-$, we associate with a sequence of points $\eta_j'=(\alpha_{j}, a_j +\epsilon_j+i b_j)$, where $\epsilon_j>0$ and $\beta_j=a_j+i b_j$, such that $\eta_j'=(\alpha_j,\beta_j')$ is in the hypersurface $\{\rho=0\}$ for every $j\in\mathbb N^*$. We note that $\epsilon_j\approx \mathrm{dist}(\eta_j,\partial \Omega)$. 

As in the proof of Theorem \ref{maintheorem1}, let us consider the sequences of  translations $L_{\eta_j'}$ and polynomial automorphisms $Q_j$ of $\mathbb C^2$ ($j\in \mathbb N^*$), defined respectively by
$$
L_{\eta_j'}(z,w):=(z,w)-\eta_j'=(z-\alpha_j,w-\beta_j')
$$
and 
\[
\begin{cases}
w&:= \tilde{w}+(R_2'(b_j) + R(\alpha_j)) i\tilde w+2\sum\limits_{k=1}^{2m}\frac{1}{k!}  \frac{\partial^k P}{\partial z^k} \tilde w^k(\alpha_j)\\
&+2\sum\limits_{k=1}^{2m}\frac{1}{k!}  \frac{\partial^k R_1}{\partial z^k} \tilde w^k(\alpha_j)+2b_j\sum\limits_{k=1}^{2m}\frac{1}{k!}  \frac{\partial^k R(\alpha_j)}{\partial z^k} \tilde w^k ;\\
z&:=\tilde{z}.
\end{cases}
\]
Then one sees that $Q_j\circ L_{\eta_j'}(\alpha_j,\beta_j)=(0,-\epsilon_j-i(R_2'(b_j) + R(\alpha_j)) \epsilon_j)$ for every $j\in \mathbb N^*$. Moreover, the hypersurface $Q_j\circ L_{\eta_j'}(\{\rho=0\}) $ is defined by an equation of the form
\begin{align*}
&\rho\left (L_{\eta_j}^{-1}\circ Q_j^{-1}( z,w)\right)= \mathrm{Re} ( w)+o(|\mathrm{Im}(w)|)+ \sum_{\substack{k+l\leq 2m\\
 k,l>0}} \frac{1}{k!l!}\frac{\partial^{k+l} P}{\partial z^k\partial \bar z^l}(\alpha_j) z^k\bar z^l \\
&+\sum_{\substack{k+l\leq 2m\\ k,l>0}} \frac{1}{k!l!}\frac{\partial^{k+l} R_1}{\partial z^k\partial \bar z^l}(\alpha_j) z^k\bar z^l+b_j \sum_{\substack{k+l\leq 2m\\ k,l>0}} \frac{1}{k!l!}\frac{\partial^{k+l} R}{\partial z^k\partial \bar z^l}(\alpha_j) z^k\bar z^l+\cdots=0,
\end{align*}
where the dots denote remainder terms. 

Now let us recall that
\begin{equation*}
\tau_{j}:=|\alpha_{j}|.\left(\dfrac{\epsilon_j}{|\alpha_{j}|^{2m}}\right)^{1/2\nu}.
\end{equation*}
and then we define an anisotropic dilation $\Delta_j$ by 
\[
\Delta_j (z,w):=\left( \frac{z}{\tau_j},\frac{w}{\epsilon_j}\right),\; j\in \mathbb N^*.
\]
This yields $\Delta_j\circ Q_j\circ L_{\eta_j'}(\alpha_j,\beta_j)=(0,-1-i(R_2'(b_j) + R(\alpha_j)))\to (0,-1)$ as $j\to\infty$. By the definition of $\tau_j$, a simple calculation shows that $\tau_{j}^{2m}=\epsilon_j.\left(\frac{\epsilon_j}{|\alpha_{j}|^{2m}}\right)^{\frac{m}{\nu}-1}\lesssim \epsilon_j$. Hence, we get the following estimates
\begin{align}\label{tau-estimate-1}
\epsilon_j^{1/2}\lesssim \tau_{j}\lesssim \epsilon_j^{1/2m}.
\end{align}
Furthermore, the hypersurface $\Delta_j\circ Q_j\circ L_{\eta_j'}(\{\rho=0\}) $ is defined by an equation of the form
\begin{equation}\label{taylor-defining-function-1}
\begin{split}
&\epsilon_j^{-1}\rho\left (L_{\eta_j}^{-1}\circ Q_j^{-1}\circ \left(  \Delta_j \right )^{-1}(\tilde z,\tilde w)\right)= \mathrm{Re} (\tilde w)+\sum_{\substack{k+l\leq 2m\\k,l>0}} \frac{1}{k!l!} \frac{\partial^{k+l} (P+R_1)}{\partial \tilde z^k\partial\overline{\tilde z}^l}(\alpha_j) \epsilon_j^{-1}\tau_{j}^{k+l} \tilde z^k\overline{\tilde z^l}\\
&+\epsilon_j^{-1} b_j\sum_{\substack{k+l\leq 2m\\k,l>0}} \frac{1}{k!l!} \frac{\partial^{k+l} R}{\partial \tilde z^k\partial\overline{\tilde z}^l}(\alpha_j) \epsilon_j^{-1}\tau_{j}^{k+l} \tilde z^k\overline{\tilde z}^l+\epsilon_j^{-1}o(\epsilon_j|\mathrm{Im}(\tilde w)|)+O(\tau(\eta_j,\epsilon_j))=0,
\end{split}
\end{equation}
where (\ref{tau-estimate-1}) yields the terms with weight greater than one equal $O(\tau(\eta_j,\epsilon_j)$. Hence, we consider only the convergence of monomials from (\ref{taylor-defining-function-1}) with weight order $\leq 1$.

Now thanks to Lemma \ref{higher-order-model}, we have
$$
 \frac{\partial^{k+l} (P+R_1)}{\partial \tilde z^k\partial\overline{\tilde z^l}}(\alpha_j)\epsilon_j^{-1}\tau_{j}^{k+l}  \to 0
$$
as $j\to \infty$ for $k+l\ne 2\nu$. For the case that $k+l=2\nu$, 
$$
\frac{\partial^{2\nu} R_1}{\partial \tilde z^k\partial\overline{\tilde z}^{2\nu-k}}(\alpha_j) \epsilon_j^{-1}\tau_{j}^{2\nu}\to 0
$$
as $j\to \infty$. Moreover, since $|\epsilon_j^{-1}b_j|\lesssim 1$, we obtain that
$$
\epsilon_j^{-1}b_j  \frac{\partial^{k+l} R}{\partial \tilde z^k\partial\overline{\tilde z}^l}(\alpha_j) \epsilon_j^{-1}\tau_{j}^{k+l} \to 0
$$
as $j\to\infty$ for any $k,l\geq 1$. In addition, since every sequence $\Big \{\frac{\partial^{2\nu} P}{\partial z^l \partial \bar z^{l'}} (\alpha_j) \epsilon_j^{-1} \tau_j^{l+l'}\Big\}_{j \ge 1}$ is bounded if $l+l'=2\nu$, hence after taking a subsequence if necessary we may assume that
$$
a_{k,2\nu-k}:=\frac{1}{k!(2\nu-k)!}\lim_{j\to\infty}\frac{\partial^{2\nu} P}{\partial \tilde z^k\partial\overline{\tilde z}^{2\nu-k}}(\alpha_j) \epsilon_j^{-1}\tau_{j}^{2\nu}, 1\leq k,l\leq n.
$$
 Therefore, after taking a subsequence if necessary, we may achieve that sequence of defining funtions given in (\ref{taylor-defining-function-1}) converges uniformly on compacta of $\mathbb C^{2}$ to $\hat\rho:=\mathrm{Re}(\tilde w)+Q(\tilde z)$. Consequently,  the sequence of domains $\Omega_j:=\Delta_j\circ Q_j\circ L_{\eta_j'}(U_0^-) $ converges normally to the following model
$$
M_{Q}:=\left \{(\tilde z,\tilde w)\in \mathbb C^2\colon \hat\rho(\tilde z,\tilde w):=\mathrm{Re}(\tilde w)+Q(\tilde z)<0\right\},
$$
where 
$$
Q(\tilde z)=\sum_{k=1}^{2\nu-1} a_{k,2\nu-k}  \tilde z^k\overline{\tilde z}^{2\nu-k}.
$$

Note that since $M_Q$ is pseudoconvex, we infer that $Q$ is subharmonic. Moreover, by the condition (iii), we have $a_{l_0,2\nu-l_0}\ne 0$ for some $1\leq l_0\leq 2\nu-1$, and hence the polynomial $Q$ is not harmonic. Furthermore, by the same argument as in the proof of Theorem \ref{maintheorem1} we conclude that $\Omega$ is biholomorphically equivalent to $M_{Q}$. Hence, the proof of Proposition \ref{1.2new} is eventually complete. \hfill $\Box$
\subsection{The Kohn-Nirenberg type domains}
The Kohn-Nirenberg domain $\Omega_{KN}$ is given firstly in \cite{KN73} by
$$
\Omega_{KN}:=\left\{(z,w)\in \mathbb C^2\colon \mathrm{Re}(w)+ |z|^8+\frac{15}{7}|z|^2\mathrm{Re}(z^6)<0\right\}
$$
This domain is a weakly pseudoconvex domain of finite type having no supporting function at the origin. In addition, $\Omega_{KN}$ cannot be biholomorphically equivalent to a bounded domain in $\mathbb C^2$ with real analytic boundary. Indeed, if this would be the case, then \cite{BP99} tells us that $\Omega_{KN}$ is biholomorphically equivalent to the ellipsoid 
$$
\left\{(z,w)\in \mathbb C^2\colon |w|^2+|z|^{8}<1\right\},
$$ 
which is biholomorphically equivalent to the model $\left\{(z,w)\in \mathbb C^2\colon \mathrm{Re}(w)+|z|^{8}<1\right\}$.  This leads to a contradition by \cite[Main Theorem]{CP01}.

Furthermore, $\Omega_{KN}$ is a $WB$-domain and thus it satisfies the hypothesis of Theorem \ref{maintheorem2}. Therefore, Theorem \ref{maintheorem2} shows that if a sequence of automorphism orbits $\{\eta_j=\varphi_j(a)\}$, $\{\varphi_j\}\subset \mathrm{Aut}(\Omega_{KN})$, converges to the origin, then it must converge $\frac{1}{8}$-nontangentially to the origin.

The condition that $\eta_j'$ is strongly pseudoconvex ensures that a limit of the scaling domains (the model $M_H$) is just biholomorphically equivalent to the unit ball $\mathbb B^2$.  If this condition is not satisfied such as $\Delta P(\alpha_j)=0$ for all $j$, then the model now becomes $M_Q=\{(z,w)\in \mathbb C^2\colon \mathrm{Re}(w)+Q(z)<0\}$ for some homogeneous subharmonic polynomial $Q$ with $\deg(Q) \geq 4$. The following example will demonstrate these phenomena. It also describes a situation when Proposition \ref{1.2new} may occur.
\begin{example} \label{KNex}
Let $\widetilde\Omega_{KN}$ be a domain in $\mathbb C^2$ defined by
$$
\widetilde\Omega_{KN}:=\left\{(z,w)\in \mathbb C^2\colon \mathrm{Re}(w)+ |z|^8-\frac{16}{7}|z|^2\mathrm{Re}(z^6)<0\right\}.
$$
We see that the sequence $\big\{(\dfrac{1}{\sqrt[8]{j}},\dfrac{9}{7j}-\dfrac{1}{j^2})\big\} $
converges $\frac{1}{8}$-tangentially but not spherically $\frac{1}{8}$-tangentially to $(0,0)$. 

 Let $\rho(z,w)=\mathrm{Re}(w)+ |z|^8-\dfrac{16}{7}|z|^2\mathrm{Re}(z^6)$ and let $\eta_j=( \dfrac{1}{\sqrt[8]{j}},\dfrac{9}{7j}-\dfrac{1}{j^2})$ for every $j\in \mathbb N^*$. This implies that $\rho(\eta_j)=\dfrac{9}{7j}-\dfrac{1}{j^2}-\dfrac{9}{7j}=-\dfrac{1}{j^2}\approx -\mathrm{dist}(\eta_j,\partial \widetilde\Omega_{KN})$. Set
$\epsilon_j=|\rho(\eta_j)|=\dfrac{1}{j^2}$. Then, a computation shows that
\begin{align*}
&\rho(z,w)\\
&=\mathrm{Re}(w)+ |(z-\frac{1}{j^{1/8}})+\frac{1}{j^{1/8}}|^8-\frac{16}{7}|(z-\frac{1}{j^{1/8}})+\frac{1}{j^{1/8}}|^2\mathrm{Re}(((z-\frac{1}{j^{1/8}})+\frac{1}{j^{1/8}})^6)\\
&=\mathrm{Re}(w)+\frac{1}{j^{8/8}}+\frac{8}{j^{7/8}} \mathrm{Re}(z-\frac{1}{j^{1/8}})+\frac{16}{j^{6/8}} |z-\frac{1}{j^{1/8}}|^2+\frac{12}{j^{6/8}} \mathrm{Re}((z-\frac{1}{j^{1/8}})^2)\\
&+\frac{8}{j^{5/8}}\mathrm{Re}((z-\frac{1}{j^{1/8}})^3)+\frac{48}{j^{5/8}}|z-\frac{1}{j^{1/8}}|^2\mathrm{Re}(z-\frac{1}{j^{1/8}})\\
&+\frac{2}{j^{4/8}} \mathrm{Re}((z-\frac{1}{j^{1/8}})^4)+\frac{32}{j^4}|z-\frac{1}{j^{1/8}}|^2\mathrm{Re}((z-\frac{1}{j^{1/8}})^2)+\frac{36}{j^{4/8}}|z-\frac{1}{j^{1/8}}|^4\\
&-\frac{16}{7}\left(\frac{1}{j^{8/8}}+\frac{8}{j^{7/8}}\mathrm{Re}(z-\frac{1}{j^{1/8}})+\frac{21}{j^{6/8}} \mathrm{Re}((z-\frac{1}{j^{1/8}})^2)+\frac{7}{j^{6/8}}|z-\frac{1}{j^{1/8}}|^2\right)\\
&-\frac{16}{7}\left(\frac{35}{j^{5/8}}\mathrm{Re}((z-\frac{1}{j^{1/8}})^3)+\frac{21}{j^{5/8}}|z-\frac{1}{j^{1/8}}|^2\mathrm{Re}(z-\frac{1}{j^{1/8}})\right)\\
&-\frac{16}{7}\left(\frac{21}{j^{4/8}}\mathrm{Re}((z-\frac{1}{j^{1/8}})^4)+\frac{35}{j^{4/8}}|z-\frac{1}{j^{1/8}}|^2\mathrm{Re}\big((z-\frac{1}{j^{1/8}})^2\big)\right)+O(\frac{1}{j^{3/8}}|z-\frac{1}{j^{1/8}}|^5)\\
&=\mathrm{Re}(w)-\frac{9}{7j^{8/8}}-\frac{72}{7j^{7/8}} \mathrm{Re}(z-\frac{1}{j^{1/8}})-\frac{36}{j^{6/8}} \mathrm{Re}((z-\frac{1}{j^{1/8}})^2)-\frac{72}{j^{5/8}}\mathrm{Re}((z-\frac{1}{j^{1/8}})^3)\\
&\quad-\frac{46}{j^{4/8}}\mathrm{Re}((z-\frac{1}{j^{1/8}})^4)-\frac{48}{j^{4/8}}|z-\frac{1}{j^{1/8}}|^2\mathrm{Re}((z-\frac{1}{j^{1/8}})^2)\\
&\quad+\frac{36}{j^{4/8}}|z-\frac{1}{j^{1/8}}|^4+O(\frac{1}{j^{3/8}}|z-\frac{1}{j^{1/8}}|^5).
\end{align*}

To define an anisotropic dilation,  let us denote by
 $\tau_j:=\tau(\eta_j)=\frac{1}{j^{3/8}}$ for all $j\in \mathbb N^*$. Now let us introduce a sequence of  polynomial automorphisms $\phi_{{\eta}_j}^{-1}$ of $\mathbb C^2$, given by
\[
\begin{cases}
 z=\dfrac{1}{\sqrt[8]{j}}+\tau_j \tilde  z\\
w=\epsilon_j \tilde  w+\dfrac{9}{7j}+\dfrac{72}{7j^{7/8}} \tau_j \tilde  z+\dfrac{36}{j^{6/8}}  \tau_j^2 \tilde  z^2+\dfrac{72}{j^{5/8}}\tau_j^3\tilde   z^3+\dfrac{46}{j^{4/8}}\tau_j^4 \tilde  z^4.
\end{cases}
\]
Therefore, we have
\begin{equation*}
\begin{split}
\epsilon_j^{-1}\rho\circ \phi_{{\eta}_j} ^{-1}(\tilde z,\tilde w)&= \mathrm{Re}(\tilde w) + 36|\tilde z|^4-48|\tilde z|^2\mathrm{Re}(\tilde z^2)
+O(\frac{1}{j^{1/4}}).
\end{split}
\end{equation*} 
We now show that there do not exist a sequence $\{f_j\}\subset \mathrm{Aut}(\widetilde\Omega_{KN})$ and $a\in \widetilde\Omega_{KN}$ such that
$\eta_j=f_j(a)\to (0,0)\in \partial \widetilde\Omega_{KN}$ as $n\to \infty$. Indeed, suppose otherwise that there exist such a sequence $\{f_j\}$ 
and such a point $a\in \widetilde\Omega_{KN}$. Then by the same argument as in the proof of Theorem \ref{maintheorem2}, $\widetilde\Omega_{KN}$ is biholomorphically equivalent to the following domain
 $$
 D:=\left\{(\tilde z,\tilde w)\in \mathbb C^2: \mathrm{Re}(\tilde w)+36 |\tilde z|^4-48|\tilde z|^2\mathrm{Re}(\tilde z^2)<0\right\}.
 $$
 However, since the D'Angelo type of $\partial D$ is always less than or is equal to $4$, it follows that $D$ is not biholomorphically equivalent to $\widetilde\Omega_{KN}$ (cf. \cite[Main Theorem]{CP01}). It is impossible.
\end{example}

\begin{acknowledgement}Part of this work was done while the first author was visiting the Vietnam Institute for Advanced Study in Mathematics (VIASM). He would like to thank the VIASM for financial support and hospitality. 
\end{acknowledgement}

\bibliographystyle{plain}

\end{document}